\documentclass{amsart}
\usepackage{graphicx} 

\usepackage[margin=1in]{geometry}
\usepackage{ytableau}
\usepackage{amsmath}
\usepackage{amsfonts}
\usepackage{amssymb}
\usepackage{amsthm}
\usepackage{youngtab}
\usepackage{tikz-cd}
\usepackage{enumitem}
\usepackage{hyperref}
\usepackage{comment}

\renewcommand{\emph}{\textbf}

\newtheorem{theorem}{Theorem}[section]
\newtheorem{proposition}[theorem]{Proposition}
\newtheorem{corollary}[theorem]{Corollary}
\newtheorem{lemma}[theorem]{Lemma}
\newtheorem{conjecture}[theorem]{Conjecture}

\newtheorem{example}[theorem]{Example}
\newtheorem{remark}[theorem]{Remark}

\theoremstyle{definition}

\newtheorem{definition}[theorem]{Definition}

\newtheorem{alg}[theorem]{Algorithm}

\newcommand{\SYT}{\mathrm{SYT}}
\newcommand{\SSYT}{\mathrm{SSYT}}

\newcommand{\cc}{\mathrm{cc}}
\newcommand{\rw}{\mathrm{rw}}

\newcommand{\sh}{\mathrm{sh}}
\newcommand{\std}{\mathrm{std}}

\newcommand{\sgn}{\mathrm{sgn}}
\newcommand{\ctype}{\mathrm{ctype}}
\newcommand{\ccTab}{\mathrm{ccTab}}
\newcommand{\xx}{\mathbf{x}}

\newcommand{\Tab}{\mathrm{Tab}}

\newcommand{\fS}{{Z}}

\newcommand{\Snake}{\mathrm{Snake}}

\newcommand{\BBB}{\mathcal{B}}

\newcommand{\dom}{\unrhd}

\title[Higher Specht Polynomials and Garsia-Procesi Modules]{Nonvanishing higher Specht polynomials and a construction for three row and hook shape Garsia--Procesi modules}

\title[Higher Specht Polynomials and Garsia-Procesi Modules]{Nonvanishing higher Specht polynomials and a construction for three row and hook shape Garsia--Procesi modules}
\author{Raymond Chou}
\address{Department of Mathematics, UC San Diego, San Diego, CA 92093}
\email{r2chou@ucsd.edu}
\author{Maria Gillespie} 
\address{Department of Mathematics, Colorado State University, Fort Collins, CO 80523, USA}
\email{maria.gillespie@colostate.edu}
\thanks{The second author was partially supported by the Simons Foundation.}
\author{Mitsuki Hanada}
\address{Department of Mathematics, University of Pennsylvania, Philadelphia, PA, 19104}
\email{mitsuki@sas.upenn.edu}

\date{\today}

\begin{document}

\begin{abstract}
    Given a polynomial ring quotient $R=\mathbb{C}[x_1,\ldots,x_n]/I$ with an action of the symmetric group induced by permuting the variables, a higher Specht basis is a collection of bases for each irreducible $\mathfrak{S}_n$-module in its decomposition that mimics the behavior of the classical Specht polynomial construction in the lowest degrees. Higher Specht bases have now been constructed for the coinvariant ring, the full polynomial ring, the rings $R_{n,k}$ appearing in the $t=0$ Delta conjecture, the hook shape Garsia-Haiman modules, and the two-row Garsia-Procesi modules, and more.

    We establish a general theory for determining when a higher Specht polynomial is nonzero, and give a proof of a conjectural higher Specht basis for all Garsia-Procesi modules in the cases of three row shapes and hook shapes.
\end{abstract}

\maketitle

\section{Introduction and Main Results}

The coinvariant ring and its generalizations are central to symmetric function theory, the geometry of flag varieties, and the representation theory of the symmetric group.  Defining the elementary symmetric polynomials $e_d(x_1,\ldots,x_n)=\sum_{1\le i_1<\cdots<i_d\le n} x_{i_1}\cdots x_{i_d}$,
the coinvariant ring in $n$ variables is $$R_n=\mathbb{C}[x_1,\ldots,x_n]/(e_1,\ldots,e_n)$$ and it comes with the natural action of the symmetric group $\mathfrak{S}_n$ by permuting the variables.  Under this action, $R_n$ behaves as a graded version of the regular representation of the symmetric group.  It also is a fundamental object in Schubert calculus, as it is isomorphic to the cohomology ring of the complex complete flag variety in $n$ dimensions.

While there are many natural bases of $n!$ elements known for $R_n$, including the Schubert polynomials \cite{Fulton}, the Artin monomial basis \cite{Artin}, and the descent monomial basis, \cite{GarsiaStanton}, none of these bases respected the decomposition of $R_n$ as an $\mathfrak{S}_n$-module into irreducible representations.  As a solution, Ariki, Terasoma, and Yamada \cite{ATY} defined a \textbf{higher Specht basis} of $R_n$, which we define for modules in general as follows, using the definition established in \cite{HigherSpechtDiagonal}:

\begin{definition}
    A \textbf{higher Specht basis} for a (graded) $\mathfrak{S}_n$ module $M$ is a basis $\mathbb{B}$ that admits a set partition $\bigsqcup \mathbb{B}_{\lambda,i}$ where $\lambda$ and $i$ represent a partition and positive integer respectively, such that:
    \begin{enumerate}
        \item Each $\mathbb{B}_{\lambda,i}$ spans a copy of an irreducible representation $V_\lambda$ in the decomposition of $M$ into irreducibles (there may be several such copies, so we distinguish with the subscript $i$),
        \item There is a bijection from $\mathbb{B}_{\lambda,i}$ to the set of ordinary (classical) Specht polynomials $g_T$ for shape $\lambda$, that preserves the $\mathfrak{S}_n$ action with respect to each basis.
    \end{enumerate} 
\end{definition}

Here, the classical Specht polynomials $g_T$ may be defined as follows.

\begin{definition}\label{def: specht}
  The \textbf{Specht polynomial} $g_T$ corresponding to a Young tableau $T$ whose entries are $1,2,3,\ldots,n$ is given by $$g_T=\prod_{c\in \mathrm{col}(T)}\prod_{\substack{i,j\in c\\ i\text{ above }j}}(x_i-x_j)$$ where $\mathrm{col}(T)$ is the set of columns of $T$. (See Figure \ref{fig:Specht}.)
\end{definition} 

\begin{figure}[b]
    \centering
   $T=\raisebox{-0.5cm}{\young(3,716,2845)}$\hspace{2cm} $g_T=(x_3-x_7)(x_7-x_2)(x_3-x_2)\cdot (x_1-x_8)\cdot (x_6-x_4)$
    \caption{A tableau $T\in \Tab(4,3,1)$ in French notation, and the Specht polynomial $g_T$.}
    \label{fig:Specht}
\end{figure}

Writing $\Tab(\lambda)$ for the set of all fillings of the boxes of the Young diagram of $\lambda$ with $1,2,\ldots,n$ in some order, one can then define the Specht module 
as 
\begin{align}\label{eq: span of spechts}
    M_\lambda=\mathrm{span}\{g_T: T\in \Tab(\lambda)\}\subseteq \mathbb{C}[x_1,\ldots,x_n],
\end{align} and we have $M_\lambda\cong V_\lambda$ with a basis given by the $g_T$ such that $T$ is a \textbf{standard Young tableau (SYT)}, in which the entries are increasing along rows and up columns. 

An  alternative definition of $g_T$ noted that $g_T$ (up to a constant) is the result of a Young idempotent operator applied to a monomial.  In particular, define $$\varepsilon_T=\sum_{\tau \in \mathcal{C}(T)}\sum_{\sigma \in \mathcal{R}(T)}(-1)^\tau \tau \sigma$$ where $\mathcal{C}(T)\subseteq \mathfrak{S}_n$ is the group of \textit{column permutations} that preserve the columns of $T$, and $\mathcal{R}(T)\subseteq \mathfrak{S}_n$ is the group of \textit{row permutations}. 
 Then it is not hard to check that, for a constant $c$, we have $c\cdot g_T= \varepsilon_T(\xx_T^S)$ where $\xx_T^r=\prod x_i^{\mathrm{row}(i)-1}$ with $\mathrm{row}(i)$ denoting the row that $i$ occurs in in $T$, indexed from bottom to top.  (Specifically, the constant $c$ is $\prod \lambda_i!$ where $\lambda$ is the shape of $T$) Ariki, Terasoma, and Yamada then generalized this construction to define \begin{equation}\label{eq:ATY} F_T^S(\mathbf{x}) = \varepsilon_T(\xx_T^S)\end{equation} where $T,S$ are a pair of standard Young tableaux of the same shape, and $\xx_T^S$ is a monomial whose subscripts are given by the entries of $T$ and whose exponents are determined by the \textit{cocharge} algorithm on the corresponding boxes in $S$.   The main result in \cite{ATY} was as follows.

 \begin{theorem}[\cite{ATY}]\label{thm: coinvariant higher specht}
     The polynomials $F_T^S(\mathbf{x})$ form a basis for the one-variable coinvariant ring $R_n=\mathbb{C}[x_1,\ldots,x_n]/(e_1,\ldots,e_n)$, where $e_i$ is the $i$-th elementary symmetric polynomial.
 \end{theorem}

  Since then, many important generalizations of the coinvariant ring have been studied from the higher Specht point of view, including: 
  \begin{itemize}
      \item The Garsia-Procesi modules $R_\mu$ \cite{GarsiaProcesi}, which correspond to the \textit{Springer fibers},  the set of flags fixed by a nilpotent matrix of Jordan type $\mu$.  Springer theory \cite{Springer1976GreenFunctions, Springer1978Construction} also shows that the top degree part of $R_\mu$ is precisely the irreducible $\mathfrak{S}_n$-module $M_\mu$, giving these modules a special role in $\mathfrak{S}_n$ representation theory. 
      \item The $t=0$ Delta modules $R_{n,k}$, for which Gillespie--Rhoades found a basis in \cite{GillespieRhoades}.
      \item The diagonal coinvariant ring $\mathrm{DR}_n$, formed by quotienting the two-variable polynomial ring $$\mathbb{C}[x_1,\ldots,x_n,y_1,\ldots,x_n]$$ by the ideal generated by the positive degree invariants under the diagonal action of $\mathfrak{S}_n$.  No higher Specht bases is yet known, but the first monomial basis has been constructed in \cite{CarlssonOblomkov}.
      \item The Garsia-Haiman modules $\mathrm{DR}_\mu$, a quotient of $\mathrm{DR}_n$ that can be thought of as the two-variable analog of $R_\mu$, and for which Gillespie found a higher Specht basis for hook shapes $\mu$ \cite{HigherSpechtDiagonal}.
  \end{itemize} 

  In the first setting above, the Garsia--Procesi modules $R_\mu$, Gillespie--Rhoades posed a conjectural basis $\BBB_\mu'$ that generalized Ariki, Terasoma, and Yamada's construction, and proved it held for two-row shapes $\mu$.  Chou and Hanada \cite{ChouHanada} found a counterexample to the Gillespie--Rhoades construction in a case in which $\mu$ has five rows, showing in fact that one of the conjectured basis elements was identically zero as a polynomial.  They then posed a new conjectural basis $\BBB_\mu$ based on catabolizability type of standard Young tableaux, and also proved it for two row shapes.  Our first main result is as follows.

  \begin{theorem}\label{thm:main1}
      The Chou--Hanada conjectural basis $\BBB_\mu$ is a basis for $R_\mu$ for all $\mu$ having at most $3$ rows, and for hook shapes $\mu$.  Moreover, $\BBB_\mu=\BBB'_\mu$ for three rows and hook shapes.
  \end{theorem}

  Indeed, the basis $\BBB_\mu$ happens to agree with the failed Gillespie--Rhoades construction in the three row case and the hook shape case, even though (as we illustrate in Section \ref{sec:equivalence}) they diverge for four rows.  This indicates that something fundamental in the structure may shift at four rows, and new methods may be needed to prove that the Chou--Hanada basis is a valid construction for all shapes.

\subsection{Results on nonvanishing of higher Specht polynomials}
 
Since the Gillespie--Rhoades construction turned out to generate polynomials that were in fact identically zero, we also explore in this paper a general theory for constructing nonvanishing higher Specht bases for copies of $M_\lambda$ in the full polynomial ring, via the following algorithm.

\begin{definition}
    A \textbf{Young symmetrized higher Specht basis} for partition shape $\lambda$ of size $n$ is a set of polynomials in $\mathbb{C}[x_1,\ldots,x_n]$ formed via the following algorithm.  
    \begin{enumerate}
        \item Fix a \textbf{generalized cocharge tableau} $S$ of shape $\lambda$, which is any filling of the boxes with  entries in $\{0,1,2,\ldots\}$.  
        \item     For each standard Young tableau $T$ of shape $\lambda$, define $$\xx_T^S=\prod_{b\in \lambda} x_{T(b)}^{S(b)}$$ where $b$ ranges over all boxes in the Young diagram of $\lambda$.
        \item  Define $F_T^S=\varepsilon_T \xx_T^S$
    for all $T$, and write $B_\lambda^S=\{F_T^S: T\in \SYT(\lambda)\}$ where $\SYT(\lambda)$ is the set of all standard Young tableaux of shape $\lambda$.
    \end{enumerate}
   Then $B_\lambda^S$ is the Young symmetrized higher Specht basis for $S$.
\end{definition}
  
Our main results on these constructions are as follows.

\begin{proposition}
  For a given $S$, if any one element of $B_\lambda^S$ is nonzero, then all elements are nonzero and they span a copy of $M_\lambda$ in the polynomial ring.
\end{proposition}

It therefore is of interest to find conditions for which this construtions yields nonzero polynomials.  To do so, we give a simplified way to compute the coefficient of any monomial in a higher Specht polynomial $F_T^S$, which allows us to check that a construction is nonzero by finding a coefficient that is nonzero.  We first note that all monomials appearing in $F_T^S$ come from ``snakings'' of $T$, defined as follows.

\begin{definition}
     Let $T\in \SYT(\lambda)$.  A \emph{$T$-snaking} is a filling $P$ with shape $\lambda$ such that if $i,j$ occur in the same column in $T$, then $i,j$ occur in different rows of $P$.     
\end{definition}

    We refer to these fillings as $T$-\emph{snaking} because of the following visualization: we may take the columns of $T$ and ``snake'' them down different paths (possibly disconnected) as in the following example:

\begin{example}\label{ex: t-snaking-1}
    Define $S,T,P$ as follows: $$S = 
    \begin{ytableau}
        0 & 1 \\
        1 & 2 & 0 \\
        2 & 1 & 3
    \end{ytableau}\hspace{1.5cm} T = 
    \begin{ytableau}
        *(pink)6 & *(cyan)8 \\
        *(pink)3 & *(cyan)4 & *(lime)7 \\
        *(pink)1 & *(cyan)2 & *(lime)5
    \end{ytableau}\hspace{1.5cm}
        P = 
    \begin{ytableau}
        *(cyan)4 & *(pink)1 \\
        *(lime)5 & *(pink)6 & *(cyan)2 \\
        *(cyan)8 & *(lime)7 & *(pink)3
    \end{ytableau}$$
Then $P$ is a $T$-snaking, and the corresponding monomial is
    $$ \xx_P^S = x_1^1 x_2^0 x_3^3 x_4^0x_5^1x_6^2x_7^1x_8^2 = x_1x_3^3x_5x_6^2x_7x_8^2.$$
\end{example}

We will define the \textbf{box action}, denoted $\cdot$, of permutations on $S$, where a permutation is thought of as a permutation on the boxes of $\lambda$ and applies to $S$ by moving its labels in the corresponding way.

\begin{definition}
    The \textbf{$\rho$-group} $H(S)$ of a tableau $S$ is the group of row permutations $\rho$ for which $\rho\cdot S=S$.
\end{definition}

\begin{definition}
    The \textbf{(alternating) $\tau$-set} $A(S)$ of a filling $S$ of a Young diagram $\lambda$ is the set of column permutations $\tau$ of the standard filling $T$ of $\lambda$, whose box action on $S$ preserves the set of elements of each row of $S$.
\end{definition}

\begin{example}\label{ex:S}
 If we set $S$ as below, and we label the boxes as shown at right: $$S=\young(100,0112) \hspace{2cm}\young(246,1357)$$ then we have $$H(S)=\{\mathrm{id, (46),(35),(46)(35)}\} \hspace{0.5cm}\text{ and }\hspace{0.5cm}A(S)=\{\mathrm{id},(12)(34),(12)(56)\}.$$
\end{example}

\begin{theorem}\label{thm:main2}
    The coefficient of $\xx_P^S$ in $F_T^S$ is $$\pm |H(S)|\cdot \sum_{\tau \in A(\pi\cdot S)}(-1)^\tau$$ where $\pi$ is the row permutation that returns each box of $P$ to its original column in $T$.
\end{theorem}

\begin{example}
  For any fixed $T$ of shape $(4,3)$, the coefficient of $\xx_T^S$ in $F_T^S$ where $S$ is the tableau in Example \ref{ex:S} is
  $$\pm 4\cdot (1+1+1)=\pm 12$$
 by the computation of $H(S)$ and $A(S)$ above.
\end{example}

As another corollary, we obtain an extension of a result in \cite{HigherSpechtDiagonal} as a nonvanishing condition.

\begin{corollary}
    If $S$ is semistandard (weakly increasing across rows, strictly up columns) then the elements of $B_\lambda^S$ are nonzero.
\end{corollary}

\subsection{Outline}  The paper is structured as follows.  In Section \ref{sec:background}, we define all of the essential combinatorial and algebraic concepts that will be used throughout the paper.  In Section \ref{sec:conditions}, we prove Theorem \ref{thm:main2} and explore the question of when a higher Specht polynomial is nonzero.  We prove most of Theorem \ref{thm:main1} in Section \ref{sec:proof}, except for the difficulty of proving that the conjectured higher Specht polynomials do not vanish modulo the Tanisaki ideal.  We dedicate Section \ref{sec:nonzero} to this more technical aspect of the proof.

\section{Preliminaries}\label{sec:background}

Throughout, we write Young diagrams of partitions in French notation, meaning that the first part corresponds to the bottom row. A \textbf{semistandard Young tableau} is a filling of the Young diagram where the entries are weakly increasing along the rows and strictly increasing along the columns. Given a partition of $n$, a  \textbf{standard Young tableau} is a semistandard tableau where the numbers $1,\dots,n$ appear exactly once.  Let $\SYT_n$ denote the collection of all standard tableaux of size $n$.

The \textbf{dominance order} (denoted $\unrhd$) on partitions of $n$ is a partial order defined by 
\[\mu\unrhd \lambda \ \  \Leftrightarrow \ \  \mu_1 + \cdots +\mu_k \geq \lambda_1+\cdots  +\lambda_k \text{ for all } k.\]

The \emph{cocharge labels} of a permutation $w$ is a word of length $n$ consisting of the labeling of $w$ that we obtain in the following way. 
Label the 1 of $w$ with $0$. Assume we labeled the letter $i$ with $k$. 
Label $(i+1)$ with a $k$ if it is to the right of $i$. We label $(i+1)$ with a $(k+1)$ if it is to the left of $i$. 
This process is equivalent to scanning $w$ from left to right, labeling entries in increasing order.
The label of $i$ corresponds to the number of times we need to return to the left end of $w$ before we label $i$. 

We define the \emph{cocharge} of $w$, denoted $\cc(w)$, to be the sum of its cocharge labels.
\begin{example}
The permutation $w =25314$ has cocharge labels $w = 2_1 5_2 3_1 1_0 4_1$. Thus we have $\cc(w) = 1+2+1+0+1 = 5$.
\end{example}
    
We can also define cocharge on word of partition content $\mu$.
The \emph{first cocharge subword} $z^{(1)}$ of $z$, where $z$ is of content $\mu$, is constructed in the following way.
Choose the rightmost 1 to be in $z^{(1)}$. 
After choosing $i$ to be in $z^{(1)}$, we choose $i+1$ in the following way.
If there is a $i+1$ to the left of the $i$ we chose, we choose the rightmost such $i+1$ to be in $z^{(1)}$. Otherwise, we choose the rightmost $i+1$ in $w$. We continue this process until we cannot add any more letters. 

Now, delete the subword $z^{(1)}$ from $z$ and repeat, continuing until we have fully decomposed $z$ into subwords $z^{(1)},\dots, z^{(\mu_1)}$. The \emph{cocharge} of $z$ is given by
\[\cc(z) := \sum_{i=1}^{\mu_1} \cc(z^{(i)})\]

The \emph{cocharge} of a tableau is the cocharge of its reading word.  The \textbf{cocharge tableau} of a tableau $S$, written $\cc(S)$, is the tableau formed by placing the subscripts of the cocharge computation on the reading word into the corresponding boxes.

\subsection{Cocharge is greedy}

We now prove a technical lemma about cocharge that does not to our knowledge appear elsewhere in the literature: that the labels on the first cocharge subword are at least as large as the corresponding labels for the second cocharge subword, and so on.  Below, we write $|z|$ for the length of a word $z$.

\begin{lemma} \label{lem:greedy}
    For a word with partition content $\mu$, define $z^{(1)},\ldots,z^{(\mu_1)}$ to be its cocharge subwords in order.  Define $\cc_k^{(i)}$ to be the cocharge subscript on the letter $k$ in $z^{(i)}$ if $|z^{(i)}|\ge k$, and define it to be $-1$ if $|z^{(i)}|<k$.  
    
    Then if $i<j$ we have $c_k^{(i)}\ge c_k^{(j)}$ for all $k$.  Moreover, if $c_k^{(i)}=c_k^{(j)}$ then the $k$ from $z^{(i)}$ appears to the right of that from $z^{(j)}$.
\end{lemma}

\begin{proof}
   We first note that it suffices to prove the lemma for the first two cocharge subwords, $z^{(1)}$ and $z^{(2)}$, and inductively the result will follow.  We prove both claims simultaneously by strong induction on $k$.  For $k=1$, all cocharge subscripts are $0$, and the $1$s are encountered in right to left order to make each $z^{(i)}$, so both claims hold.

   For the induction step, let $n\ge 1$ and assume the claim is true for all $k\le n$.  We show the claim holds for $k=n+1$.  If $z^{(1)}$ contains $n$ but $z^{(2)}$ does not, then $c_{n+1}^{(2)}=-1<c_{n+1}^{(1)}$ and we are done.  So we may now assume that $n+1$ appears in both cocharge words. 
    
   Write $\bf{n}$ and $n$ respectively, and $\bf{n+1}$ and $n+1$ respectively, for the copies of $n$, and $n+1$, in $z^{(1)}$ and $z^{(2)}$ respectively.  Let $i=c_n^{(1)}$ and $j=c_n^{(2)}$, and $a=c_{n+1}^{(1)}$ and $b=c_{n+1}^{(2)}$ be the cocharge subscripts on these letters.     
   We wish to show that $a\ge b$, and that if $a=b$ then $\bf{n+1}$ is to the right of $n+1$.

    There are $12$ cases to consider for how $n,{\bf n},n+1,\bf{n+1}$ appear relative to each other in order.  Indeed, there are $24$ permutations of $4$ different letters, but the choice of which $n$ is boldface forces the choice of which $n+1$ is boldface by the cocharge algorithm.  We list these cases below, and calculate $a,b$ in terms of $i,j$ based on the cocharge algorithm:
\begin{center}
\begin{tabular}{c|c|c|c}
    Case & Ordering & $a$ & $b$  \\\hline
    1 & $n_j{\bf n}_i (n+1)_b({\bf n+1})_a$ & $i$ & $j$ \\\hline
    2 & ${\bf n}_i n_j(n+1)_b({\bf n+1})_a$ & $i$ & $j$ \\\hline
    3 & $n_j({\bf n+1})_a{\bf n}_i (n+1)_b$ & $i+1$ & $j$ \\\hline
    4 & ${\bf n}_i (n+1)_b n_j({\bf n+1})_a$ & $i$ & $j+1$ \\\hline
    5 & $n_j(n+1)_b ({\bf n+1})_a{\bf n}_i$ & $i+1$ & $j$ \\\hline
    6 & ${\bf n}_i (n+1)_b ({\bf n+1})_a n_j $ & $i$ & $j+1$ \\\hline
    7 & $({\bf n+1})_an_j{\bf n}_i (n+1)_b$ & $i+1$ & $j$ \\\hline
    8 & $({\bf n+1})_a{\bf n}_in_j (n+1)_b$ & $i+1$ & $j$ \\\hline
    9 & $(n+1)_bn_j({\bf n+1})_a{\bf n}_i $ & $i+1$ & $j+1$ \\\hline
    10 & $({\bf n+1})_a {\bf n}_i (n+1)_bn_j$ & $i+1$ & $j+1$ \\\hline
    11 & $(n+1)_b ({\bf n+1})_an_j{\bf n}_i $ & $i+1$ & $j+1$ \\\hline
    12 & $(n+1)_b ({\bf n+1})_a{\bf n}_i n_j$ & $i+1$ & $j+1$ 
\end{tabular}
\end{center}

We first prove that $a\ge b$.  In all cases above except Case 4 and Case 6, this is immediate by the induction hypothesis, which tells us $i\ge j$.  In Cases 4 and 6, the boldface $\bf n$ is to the left of $n$, so by the induction hypothesis, we have the strict inequality $i>j$.  Thus, even though $a=i$ and $b=j+1$, we still have $a\ge b$. 

We now prove that if $a=b$ then the boldface $\bf n+1$ is to the right of $n+1$ in the word.  Since $\bf n+1$ is to the right of $n+1$ in Cases 1, 2, 4, 5, 6, 9, 11, and 12, we are done in these cases.  In Cases 3, 7, and 8, we have that $a=i+1$ and $b=j$, so since $i\ge j$, we have $a> b$ in these cases.  This leaves Case 10, in which $i>j$ strictly by the induction hypothesis because $\bf n$ is to the left of $n$.  But since $a=i+1$ and $b=j+1$ we conclude $a>b$ strictly as well.

This completes the induction.
\end{proof}

\subsection{Catabolizability type and Blasiak's method}
To each standard tableau $S$, we can associate a partition to it called its \emph{catabolizability type} $\ctype(S)$. This partition is defined using an operation called catabolism on tableaux. 
For more on catabolism, see \cite{Lascoux, SW,hanada}.

Here, we omit the precise definitions of catabolisms and catabolizability type and instead focus on how to compute it.
Blasiak \cite{blasiak} gives a method to compute the catabolizability type of a standard tableau $S$ by looking at its reading word.

 \begin{alg}[\cite{blasiak}]\label{alg: ctype}
   Let the cocharge subscripts of $\rw(S)$ be given by the word $c=c_1\cdots c_n$. 
  We read $c$ from right to left cyclically, starting with $c_n$, to construct a Young diagram with boxes in rows $0,1,2,\ldots$, starting with the empty partition, as follows: 
  \begin{itemize}
      \item If adding a box to the end of row $c_i$ results in a partition shape, then add this box and cross off $c_i$.
      \item Otherwise, increment $c_i$ by $1$ and move to the next letter $c_{i-1}$ (or the next non crossed off letter $c_j$ cyclically to the left).  
  \end{itemize}  
  The final partition shape is $\ctype(S)$.
 \end{alg}

\begin{example}
    Consider \ytableausetup{boxsize  = 1em} $S = \ytableaushort{34,125}$ with $\rw(S) = 34125$. The corresponding cocharge subscripts are given by the word $11001$. Applying Algorithm \ref{alg: ctype} results in the partition $(2,2,1)$:
         \ytableausetup{smalltableaux, centertableaux}
     {\allowdisplaybreaks\begin{align*}
          ( 1 \ 1 \ 0 \ 0 \ \mathbf{\underline{1}} \hspace{.1cm}&, \hspace{.3cm} \emptyset) \\ &\downarrow\\ 
     ( 1 \ 1 \ 0 \ \mathbf{\underline{0}} \ 2\hspace{.1cm} &, \hspace{.3cm}\emptyset \hspace{.1cm}) \\ &\downarrow\\ 
     ( 1 \ 1 \  \mathbf{\underline{0}} \ \hspace{.2cm} \ 2 \hspace{.1cm}&, \hspace{.3cm} \ytableaushort{4} \hspace{.1cm}) 
      \\ &\downarrow\\ 
     (1 \ \mathbf{\underline{1}} \  \hspace{.2cm} \ \hspace{.2cm} \ 2 \hspace{.1cm}&,\hspace{.1cm} \ytableaushort{43}\hspace{.1cm}) \\ &\downarrow\\ 
     ( \mathbf{\underline{1}} \  \hspace{.2cm}  \  \hspace{.2cm} \ \hspace{.2cm} \ 2\hspace{.1cm} &, \hspace{.3cm}\ytableaushort{2,43}\hspace{.1cm}) 
     \\ &\downarrow\\ 
     (\hspace{.2cm} \  \hspace{.2cm}  \  \hspace{.2cm} \ \hspace{.2cm} \ \mathbf{\underline{2}}  \hspace{.1cm}&,\hspace{.3cm} \ytableaushort{21,43}\hspace{.1cm}) 
 \\ &\downarrow\\ 
     (  \hspace{.2cm}  \ \hspace{.2cm}  \  \hspace{.2cm} \ \hspace{.2cm} \ \hspace{.2cm} \hspace{.1cm} &, \hspace{.3cm}\ytableaushort{5,21,43}\hspace{.1cm}) 
    \end{align*}}
    From this, we conclude $\ctype(S) = \ctype(34125) = (2,2,1)$.
\end{example}
\subsection{Garsia--Procesi modules}

Given a partition $\mu$, let $\mu' = (\mu_1',\mu_2',...,0)$ denote its transpose, padded with $0$'s to form a composition of length $n$ and set $p^n_m(\mu) := \mu_n' + \dots + \mu'_{n-m+1}$. Given a subset $Y \subset [n]$ and $d \leq |Y|$, we may define
\[
e_d(Y) := \sum_{i_i,\dots,i_d \in Y} x_{i_1}\dots x_{i_d}
\]
The \emph{Tanisaki ideal} is the ideal generated by the generalized elementary symmetric polynomials
\[ 
I_\mu := ( e_d(Y) : Y \subset [n], d > |Y| - p^n_{|Y|}(\mu)).
\]
Note that we have $\lambda \unrhd \mu$ if and only if $I_\mu \subset I_\lambda$.

The Garsia-Procesi module is defined to be the quotient
\[
R_\mu := \frac{\mathbb{C}[x_1,\dots,x_n]}{I_\mu}
.\]

The Garsia-Procesi module is well known to be isomorphic to $H^*(\mathrm{Sp}_\mu)$, the ordinary cohomology ring of the nilpotent Springer fiber $\mathrm{Sp}_\mu$. Since the ideal $I_\mu$ is $\mathfrak{S}_n$-stable, there is a natural $\mathfrak{S}_n$-action on $R_\mu$ defined by permuting the variables $x_1,\dots,x_n$, which coincides with the well-known \emph{Springer action}. The graded character was first computed by Springer \cite{Springer1978Construction} and later Garsia-Procesi \cite{GarsiaProcesi} showed the graded multiplicites are given by the \emph{Kostka-Foulkes polynomials}:
 
\[
\widetilde{K}_{\lambda\mu}(q) = \sum_{k} q^k \langle \chi^\lambda, \mathrm{char} (R_\mu)_k \rangle
\]
where $\chi^\lambda$ is the irreducible character of $\mathfrak{S}_n$ indexed by $\lambda$ and $(R_\mu)_k$ denote the degree $k$ component of $R_\mu$. The polynomials $\widetilde{K}_{\lambda\mu}(q)$ have a famous description in terms of the \emph{cocharge} statistic of Lascoux-Schutzenberger \cite{LascouxSchutzenberger1978Foulkes}.

\[
\widetilde{K}_{\lambda\mu}(q) = \sum_{T \in \SSYT(\lambda,\mu)} q^{\cc(T)}
\]
where $\SSYT(\lambda,\mu)$ is the set of set of semistandard Young tableaux with shape $\lambda$, content $\mu$. Under Lascoux standardization, the polynomials $\widetilde{K}_{\lambda\mu}(q)$ have a description in terms of catabolizability type:

\[
\widetilde{K}_{\lambda\mu}(q) = \sum_{\substack{T \in \SYT(\lambda) \\ \ctype(T) \trianglerighteq \mu}} q^{\cc(T)}
\]

There is an alternative description of $I_\mu$ in terms of harmonics and Specht polynomials. Note that we can define Specht polynomials for any filling $\widetilde{T}\in \Tab(\lambda)$ in the same way:

\[g_{\widetilde{T}}=\prod_{c\in \mathcal{C}(\widetilde{T})}\prod_{\substack{i,j\in c\\ i\text{ above }j}}(x_i-x_j).\]

We make a quick observation:
\begin{lemma}\label{lemma: col strict garnirs}
Let $\widetilde{T}\in \Tab(\lambda)$. We have $g_{\widetilde{T}} \in M_{\lambda}$.
\end{lemma}

\begin{proof}
Since $M_{\lambda}\cong V_{\lambda}$, for any $T\in \SYT(\lambda)$ and permutation $\sigma$, we have that $\sigma g_T \in M_{\lambda}$.
Let $\pi$ be any permutation such that $\pi(\widetilde{T})\in \SYT(\lambda)$. Thus $g_{\widetilde{T}} = \pi^{-1} g_{\pi \widetilde{T}}\in M_{\lambda}$.
\end{proof}

Now, for any polynomial $f\in \mathbb{C}[x_1,\dots, x_n]$, let $f(\partial)$ denote the differential operator we obtain by evaluating $x_i\mapsto \partial_i$. 

\begin{lemma}[Bergeron-Garsia \cite{BergeronGarsia1992HarmonicPolynomials}]\label{thm: BG}
    Let $\mathbb{C}[\partial] := \mathbb{C}[\partial/\partial x_1,\dots,\partial/\partial_n]$ denote the algebra of operators which acts on $\mathbb{C}[\xx_n]$ by differentiation. Then, $f \in I_\mu$ if and only if $f(\partial)  g_T = 0$ for all $T \in \SYT(\mu)$.
\end{lemma}

In particular, we have the following corollary using  Lemma \ref{lemma: col strict garnirs}.
\begin{corollary}\label{cor: nonzero col strict}
    If $f(\partial)g_{\widetilde{T} }\neq 0$ for some filling $\widetilde{T}\in\Tab(\lambda)$, then $f\not\in I_\lambda$.  
\end{corollary}

\subsection{Two conjectural higher Specht bases} We now recall the definitions of $\BBB_\mu$ and $\BBB_\mu'$.  
The key construction for the Chou--Hanada conjectural basis $\BBB_\mu$ is the following.

\begin{definition}
   Define $\SYT^{\unrhd \mu}=\{S\in \SYT_n: \ctype(S)\unrhd \mu\}$ to be the set of all standard Young tableaux whose catabolizability type weakly dominates $\mu$.
\end{definition}

Using the inclusion $\SYT^{\unrhd \mu}\subseteq \SYT_n$, we can identify a subset of the higher Specht basis of the coinvariant ring $R_n$.
\begin{conjecture}[\cite{ChouHanada}]\label{conj: new higher specht}
Let $\BBB_\mu:=\{F_T^S: S\in \SYT^{\unrhd \mu}, \sh(S) = \sh(T)\}$ be the subset of the higher Specht basis of $R_n$ indexed by $S\in \SYT^{\unrhd\mu}$.
The set $\BBB_\mu$ is a higher Specht basis of $R_\mu$.
\end{conjecture}

On the other hand, the Gillespie--Rhoades construction is in terms of semistandard Young tableaux of a given content $\mu$, which we denote $\SSYT_\mu := \bigsqcup_{\lambda} \SSYT(\lambda,\mu)$.

\begin{proposition}[\cite{Lascoux}] \label{prop: lascoux}
    For any $\mu\vdash n$, we have a bijection 
    $\phi: \SSYT_\mu\rightarrow \SYT^{\unrhd \mu}$
    which preserves cocharge and shape.
\end{proposition}

The bijection in Proposition \ref{prop: lascoux} is in terms of embeddings for cyclage posets: for an explicit description of the bijection, see \cite{SW}.

\begin{conjecture}[\cite{GillespieRhoades}]\label{conj:GillespieRhoades}
    For any pair $(T,S)$ where $S\in \SSYT_{\mu}$ and $T$ is any standard Young tableau of the same shape as $S$, define $$F_T^S=\varepsilon_T \xx_T^{\cc(S)}.$$  Then the set $\BBB'_\mu$ of all such polynomials $F_T^S$ is a higher Specht basis of $R_\mu$.
\end{conjecture}

\subsection{Counterexample for Gillespie-Rhoades}  \label{sec:counterexample}

We recall the example from \cite{ChouHanada} that shows that the conjectured basis $\BBB_\mu'$ is not a basis; in fact, the constructed polynomials are not even always nonzero, so Conjecture \ref{conj:GillespieRhoades} is false.  As an example, for $\mu=(2,2,2,2,2)$, consider
\[\raisebox{0.7cm}{$S =$} \,\,\,\,\young(5,3,2345,1124)\hspace{1cm} \raisebox{0.7cm}{$\cc(S)=$} \,\,\,\, \young(2,2,1113,0002)\]
For any choice of $T$,  the terms in $F_T^S = \varepsilon_T x_T^{\cc(S)}$ will cancel in pairs, by applying the transposition that swaps the two top entries in the first column (which have the same exponent from $\cc(S)$). Thus $F_T^S=0$ identically.

\subsection{Equivalence with Gillespie-Rhoades construction for three rows }\label{sec:equivalence}

We now show that the two basis constructions agree for three row shapes $\mu$.   To do so, it suffices to construct a cocharge tableau preserving bijection from the set $\SSYT_\mu$ of semistandard Young tableaux of content $\mu=(\mu_1,\mu_2,\mu_3)$ to the set $\SYT^{\unrhd\mu}$ of standard Young tableaux whose catabolizability type dominates $\mu$.  We define $f:\SSYT_\mu\to \SYT^{\unrhd \mu}$ by $$f(T)=\std(\cc(T)),$$ 
where $\mathrm{std}$ is the classical standardization map on semistandard tableaux (i.e, we standardize letters from left to right). We show this is a well defined bijection that preserves the cocharge tableau.

\begin{remark}
 One can check that the map $f$ in fact agrees with the Lascoux bijection mentioned in Proposition \ref{prop: lascoux}, for three row shapes
\end{remark}

\begin{example}\label{ex:3-row}
    The map $f$ in the proof above maps the tableau at left to the tableau at right, and both have the same cocharge tableau shown in the middle:

$$\young(3,2223,1111122333) \hspace{1cm} \young(2,1111,0000000011) \hspace{1cm}\young({{15}},9{{10}}{{11}}{{12}},12345678{{13}}{{14}})$$
\end{example}

\begin{lemma}\label{lem:1}
    The map $f$ is well defined.
\end{lemma}

\begin{proof}
  Let $T\in \SSYT_\mu$ where $\mu$ has three rows, and suppose $T$ has shape $\lambda$.  We first show that $\cc(T)$ is semistandard.  Let $m_i(r)$ be the number of $i$'s in row $r$.  For instance, $m_1(1)$ is the number of $1$s on the bottom row.  Then by semistandardness, we have $m_1(1)=\mu_1$ and $m_1(2)=m_1(3)=0$.  We also have $m_2(1)+m_2(2)=\mu_2$,  $m_2(3)=0$, and $m_3(3)\le m_2(2)\le m_1(1)$.  
    
    Thus when we compute $\cc(T)$, the 3s on the top row have cocharge label 2, and the 2s in the second row have cocharge label 1.  For a 3 on the second row, whether the 2 that preceded it in its cocharge word was in the first row (with cocharge 0) or the second row (with cocharge 1), the 3 will have a cocharge label of 1.  Thus the entire second row of $\cc(T)$ is labeled $1$, and the entire top row is $2$. (See Example \ref{ex:3-row}.) 

    On the bottom row, all the $2$s on the bottom row have cocharge $0$ since they follow a $1$ to their left, and the $3$s have cocharge $1$ if and only if they follow a $2$ from the middle row.  This happens if and only if the $2$ did not precede a $3$ from the top row, so the number of cocharge labels $1$ at the end of the bottom row of $\cc(T)$ is $\min(m_1(3),m_2(2)-m_3(3))$. In particular, $\cc(T)$ is semistandard.  It follows that $f(T)=\std(\cc(T))$ is a well defined tableau.
    
    We finally show $f(T)\in \SYT^{\unrhd \mu}$.  Let $k$ be the number of $0$s in the bottom row of $\cc(T)$.  Then $f(T)=\std(\cc(T))$ has $1,\ldots,k$ in the bottom row, then $k+1,\ldots,k+j$ in the horizontal strip consisting of the second row and the rest of the entries on the bottom row, and finally $k+j+1,\ldots,n$ along the top row.      Notice that $\cc(f(T))=\cc(T)$ as well.  So, when we compute $\ctype(f(T))$ by Algorithm \ref{alg: ctype}, we first increment the $1s$ from the bottom row and then place exactly $k$ boxes in row $0$.  Then we place $\lambda_2$ boxes in the second row when we read the corresponding $1$s, and continue from there.  Thus the first part of $\ctype(f(T))$ is at least $k\ge \mu_1$, and the first two parts have total size at least $k+\lambda_2\ge \mu_1+\mu_2$, so $\ctype(f(T))\unrhd\mu$.
\end{proof}

\begin{lemma}\label{lem:2}
    The map $f$ preserves cocharge tableaux.
\end{lemma}

\begin{proof}
    In the above proof, we saw that for $T\in \SSYT_\mu$, the cocharge tableau $\cc(T)$ is semistandard.  Moreover, since it is a cocharge tableau, we cannot have all $i+1$s occurring after all $i$s in reading order for any $i$, because in the first cocharge subword of $T$, the first cocharge label $i+1$ in the subword must precede the last $i$ in the subword, in reading order in $T$.
    
    Thus, when we take its standardization to form $f(T)=\std(\cc(T))$, we replace the horizontal strip of $0$s with $1,2,3,\ldots,k$ in reading order, then the horizontal strip of $1$s with $k+1,\ldots,k+\ell$, and so on, where the $k+1$ will be to the left of $k$ in reading order, and the $k+\ell+1$ will be left of $k+\ell$, and so on.  Thus when we take the cocharge tableau of the resulting standard tableau, we exactly recover $\cc(T)$ by the cocharge algorithm.  
\end{proof}

\begin{lemma}\label{lem:3}
    The map $f$ is a bijection.
\end{lemma}

\begin{proof}
To show that $f$ is injective, first note that by Lemma \ref{lem:2}, we cannot have $f(T)=f(T')$ and $\cc(T)\neq \cc(T')$, since $f$ preserves cocharge tableau.  Thus it suffices to show that $\cc$ is injective on $\SSYT_\mu$, that is, that $T\in \SSYT_\mu$ is determined by its cocharge tableau, for $\mu$ having three rows.

    With the same notation for $T$ and its multiplicities above, let $T'\in \SSYT_\mu$ and suppose $\cc(T')=\cc(T)$.  Then $T'$ has all $\mu_1$ of its $1$s in the bottom row by semistandardness, its top row has all $3$s, the boxes directly below these $3$s are filled with $2$, and all of the $\lambda_1-k$ boxes in the bottom row which have $1$ in $\cc(T)$ must be filled with $3$ in $T'$ in order to attain the subscripts $1$ in $\cc(T')$.  Moreover, we must place at least $\lambda_1-k$ extra $2$s in the second row in order to ensure the $3$s do have cocharge value $1$. 
    
    For the remaining $2$s and $3$s, note that if we place another $3$ in the bottom row and another $2$ in the top, we will create an extra $1$ in the cocharge tableau, a contradiction.  Thus we must place as many $2$s as possible in the bottom row, either finishing the labeling in the bottom row or running out of $2$s, whichever comes first, and then the remaining $3$s must be placed in the remaining squares.  Thus $T'=T$ as the cocharge tableau has completely determined the tableau it came from.

    Since $f$ is injective and the sets have the same size by the Lascoux bijection, $f$ is bijective, as desired.  Since it preserves the cocharge tableau, the set of higher Specht bases are identical as well.
\end{proof}

Lemmas \ref{lem:1}, \ref{lem:2}, and \ref{lem:3} combine to yield our result.

\begin{corollary}
    The construction $\BBB'_\mu$ of Gillespie--Rhoades coincides with $\BBB_\mu$ for $\mu$ having three rows.
\end{corollary}

\begin{proposition}
    The sets $\BBB_\mu'$ and $\BBB_\mu$ do not agree for the four row shape $\mu=(2,2,2,2)$.
\end{proposition}

\begin{proof}
The tableau $$\young(4,24,11233)$$ which is semistandard of content $\mu$ has cocharge tableau
$$\young(1,12,00001)$$ which is not semistandard, but is the $S$ of a basis element in $\BBB'_\mu$.  It therefore cannot be the cocharge tableau of a standard tableau, and so it does not match the elements of $\BBB_\mu$.
\end{proof}

\subsection{Equivalence for hook shapes}

We now show that the two basis constructions agree for hook shapes, where we recall that a hook shape is a partition of the form $(k,1,1,1,\ldots,1)=(k,1^{n-k})$ for some $k$.

\begin{proposition}
    The construction $\BBB_\mu'$ of Gillespie--Rhoades coincides with $\BBB_\mu$ for all hook shapes $\mu$.
\end{proposition}

\begin{proof}
    We construct a cocharge tableau preserving bijection from the set $\SSYT_\mu$ of semistandard Young tableau of content $\mu=(k,1^{n-k})$, to the set $\SYT^{\unrhd\mu}$ of standard Young tableaux whose catabolizability type dominates $\mu$.  Indeed, we claim that the standardization map is such a bijection.

    To show $\std:\SSYT_\mu\to \SYT^{\unrhd\mu}$ is well defined, we note that all of the $1$s of any $T\in \SSYT_\mu$ appear in the bottom row (and then there is one of all of the other letters).  When we standardize, the numbers $1,2,\ldots,k$ therefore all appear in the bottom row.  This is precisely the condition to be in $\SYT^{\unrhd\mu}$.  We can also therefore destandardize any $S\in \SYT^{\unrhd \mu}$ by changing $1,2,\ldots,k$ to $1$ and decrementing all other labels, and so $\std$ is a bijection.

    Finally, the cocharge tableau is unchanged under $\std$, since in both cases the first $k$ entries on the bottom row are labeled $0$, and then the cocharge is computed as normal starting from the $k$th letter.
\end{proof}

\section{Conditions for higher specht polynomials vanishing}\label{sec:conditions}

We now explore when higher Specht constructions are nonzero as polynomials.  First, we give a combinatorial construction to track the monomials of $\varepsilon_T x_T^S$.  A \textbf{filling} is a tableau that is not necessarily semistandard; that is, there is simply a nonnegative integer in each box of the diagram with no other restrictions.

\begin{definition}
    Given a filling $P$ of $\lambda$ with content $(1,1,\ldots,1)$ and any filling $\fS$ of the same shape $\lambda$, we define the monomial $$x_P^\fS=\prod_{u \in \lambda} x_{P(u)}^{\fS(u)}$$ where the product ranges over all boxes $u$ of $\lambda$ and $P(u)$ denotes the entry in box $u$.
\end{definition}

\begin{definition}
     Let $T\in \SYT(\lambda)$.  A \emph{$T$-snaking} is a filling $P$ with shape $\lambda$ such that if $i,j$ occur in the same column in $T$, then $i,j$ occur in different rows of $P$.    We write $\Snake(T)$ for the set of all $T$-snakings. 
\end{definition}

    We refer to these fillings as $T$-\emph{snaking} because of the following visualization: we may take the columns of $T$ and ``snake'' them down different paths (possibly disconnected) as in the following example:

\begin{example}\label{ex: t-snaking}
    Let $$\fS = 
    \begin{ytableau}
        0 & 1 \\
        1 & 2 & 0 \\
        2 & 1 & 3
    \end{ytableau}\hspace{0.5cm}\text{and} \hspace{0.5cm} T = 
    \begin{ytableau}
        *(pink)6 & *(cyan)8 \\
        *(pink)3 & *(cyan)4 & *(lime)7 \\
        *(pink)1 & *(cyan)2 & *(lime)5
    \end{ytableau}\hspace{0.5cm}.$$
 An example of a $T$-snaking is
    $$P = 
    \begin{ytableau}
        *(cyan)4 & *(pink)1 \\
        *(lime)5 & *(pink)6 & *(cyan)2 \\
        *(cyan)8 & *(lime)7 & *(pink)3
    \end{ytableau}.$$
 The corresponding monomial is
    $$ x_P^\fS = x_1^1 x_2^0 x_3^3 x_4^0x_5^1x_6^2x_7^1x_8^2 = x_1x_3^3x_5x_6^2x_7x_8^2.$$
\end{example}

\begin{definition}
    For a standard tableau $T$, let $T_i$ denote the set of entries in the $i$th column of $T$. Suppose $T_i = \{ a_1 < \dots < a_k\}$. Given a $T$-snaking $P$, suppose the entries of $T_i$ appear in the bottom to top order $a_{\tau(1)},\ldots,a_{\tau(k)}$.  This permutation $\tau$ is uniquely determined and we call it $\tau^{(i)}(P)$. 
\end{definition}

\begin{definition}\label{def: snaking-sign}
    We define the \emph{sign} of a $T$-snaking $P$, denoted $(-1)^P$, to be 
    $$ (-1)^P := \prod_{i}\sgn(\tau^{(i)}(P)).$$
\end{definition}

\begin{lemma}
    Let $Z$ be a filling of shape $\lambda$ and $T\in \SYT(\lambda)$. Then
    $$ F_T^\fS=\varepsilon_T x_T^\fS = \sum_{P\in \Snake(T)} (-1)^Px_P^\fS.$$
\end{lemma}

\begin{proof}
First, consider a term $(-1)^\tau\tau \sigma x_T^\fS$ in $\varepsilon_T x_T^\fS$.  We construct an associated $T$-snaking as follows. The permutation $\sigma$ is a row permutation of $T$ and acts on the subscripts of $x_T^\fS$; let $Q$ be the tableau formed by applying the row permutation $\sigma$ to $T$.  This forms a snaking of the columns of $T$ such that each snake still has the entries in the same relative order top to bottom as in $T$.  Then we have $(-1)^\tau\tau \sigma x_T^\fS=(-1)^\tau \tau x_Q^\fS$.

Now, applying $\tau$ to $Q$, we are permuting the entries within each snake, and obtain a $T$-snaking $P$.  Since $\tau$ acts on the subscripts of $x_Q^\fS$ correspondingly, the monomial we obtain is $x_P^\fS$, and notice that $(-1)^\tau$ is the product of the signs of each individual permutation of each column that comprises $\tau$, which is equal to $(-1)^P$.  Thus $(-1)^\tau \tau \sigma x_T^\fS=(-1)^P x_P^\fS$.  We now have a way of constructing a term $(-1)^P x_P^\fS$ for each term of $\varepsilon_T x_T^\fS$.  To show that this process is reversible, note that any $T$-snaking has a unique associated column permutation $\tau^{-1}$ that straightens the columns, and then a unique row permutation $\sigma^{-1}$ that returns us to $T$; thus the correspondence is bijective and the two sums are equal.
\end{proof}
We will make use of the following lemma.
\begin{lemma}\label{lemma: pigeonhole}
In any $T$-snaking $P$, an element $b$ of a column of height $h$ in $T$ cannot be in a row higher than $h$ in $T$.
\end{lemma}
\begin{proof} 
Let $\lambda$ denote the shape of $T$.  Assume for contradiction there exists a $b$ of a column of height $h$ which is in row $h'$ ($h'>h)$ in some $T$-snaking $P$. 

Assume that $b$ is the leftmost such entry in $T$. Then we know that for any column of height $\geq h'$, there exists exactly one element from the column that is in row $h'$ in $P$. That is, there are $\lambda_{h'}$ many entries in row $h'$ of $P$ coming from columns of height $\geq h'$ in $T$. Thus, it is impossible for $b$ to also be in row $h'$.
\end{proof}

We now analyze how to group terms of this summation together to obtain the coefficient of any given monomial.  To do so, we need to define a different type of $\mathfrak{S}_n$ action on tableau - rather than acting on the labels, we permute the positions of the labels, with the boxes labeled $1,2,3,\ldots,n$ in ascending order up each column from left to right by default. We make this rigorous as follows.

\begin{definition}
    The \textbf{canonical box ordering} of a partition shape $\lambda$ is the ordering of the boxes $1,\ldots,n$ where the first column's boxes are indexed $1,2,\ldots,\lambda_1'$, the second column's boxes are $\lambda_1'+1,\ldots,\lambda_1'+\lambda_2'$, and so on, where $\lambda'$ denotes the transpose partition.

    Given a filling $\fS$ of shape $\lambda$ and $b\in \{1,\ldots,n\}$, we write $\fS(b)$ to denote the entry in the box whose index is $b$ under the canonical box labeling.
\end{definition}

\begin{definition}
    The \textbf{box action} of $\mathfrak{S}_n$ on the set of all fillings of a partition of size $n$, denoted $\cdot$, is defined by the rule $\pi \cdot \fS(b)=\fS(\pi^{-1}(b))$. 
\end{definition}

\begin{example}
Suppose we have
 $$\fS=\raisebox{-0.5cm}{\young(20,123,1142)\hspace{1cm} \young(36,258,1479)}$$ where the canonical box ordering is shown at right.  Then the permutation $(12)(59)$ fixes $\fS$, and the permutation $(568)$ acts on $\fS$ via the box action as follows:
 $$(568)\cdot \fS=\raisebox{-0.5cm}{\young(22,130,1142)}$$
\end{example}
\begin{lemma}\label{lem:row-perm-same-sign}
    Suppose $\pi$ is a row permutation of the canonical box ordering of $\lambda$. Then,
    $$\sum_{P \in\Snake(T)} (-1)^Px_P^\fS = \sum_{P \in\Snake(T)} (-1)^P x_P^{\pi \cdot \fS}.$$
\end{lemma}

\begin{proof}
  We have $x_P^\fS=x_{\pi\cdot P}^{\pi\cdot \fS}$ because we are simply rearranging the positions of the labels of each of $P$ and $\fS$ in the same way.  Notice that $\pi$, under the box action, permutes $\Snake(T)$.  Thus $$\sum_{P\in\Snake(T)} (-1)^Px_P^\fS = \sum_{P\in\Snake(T)} (-1)^P x_{\pi\cdot P}^{\pi \cdot \fS}=\sum_{P\in\Snake(T)} (-1)^P x_P^{\pi \cdot \fS}$$ as desired.
\end{proof}

We now need one more piece of notation. Intuitively, given a monomial $x_P^\fS$, we can ``straighten'' the snakes, and observe the effect that the same box permutation has on $S$.

\begin{definition}\label{def: sp}
    Let $\fS,T$ be as before (with shape $\lambda$), and suppose $P$ is a $T$-snaking. Let $T_P$ denote the unique filling of $\lambda$ with identical column content to $T$, and the order of entries appearing top to bottom identical to $P$. Noting that $T_P = \pi \cdot P$ for some $\pi \in \mathcal{R}(P)$, we define
    $$\fS_P := \pi\cdot \fS.$$
\end{definition}

\begin{example}\label{ex: t-snaking-straight}
    For the $T$-snaking in example \ref{ex: t-snaking}, we have that
    $$\fS_P = 
    \begin{ytableau}
        1 & 0 \\
        2 & 0 & 1 \\
        3 & 2 & 1
    \end{ytableau}\hspace{0.5cm}\text{and} \hspace{0.5cm} T_P = 
    \begin{ytableau}
        *(pink)1 & *(cyan)4 \\
        *(pink)6 & *(cyan)2 & *(lime)5 \\
        *(pink)3 & *(cyan)8 & *(lime)7
    \end{ytableau}\hspace{0.5cm}.$$  We also have that $(-1)^{T_p}=-1$.
\end{example}

Finally, if we want two $T$-snakings $P,P'\in \Snake(T)$ such that $x_P^\fS=x_{P'}^\fS$, the elements of $\fS$ assigned to each label in $P'$ must match those of $P$.  We claim that for fixed $P$,  the $T$-snakings $P'\in \Snake(T)$ that satisfy this are precisely in bijection with the column permutations $\tau$ of $T$ that send $\fS_P$ to a row permutation of itself under the box action.

\begin{lemma}\label{lem: same-monomial-rowcol}
For any given $T$-snaking $P$, there is a bijection between the $T$-snakings $P'$ for which $x_{P'}^\fS=x_{P}^\fS$, and pairs $(\tau,\rho)$ where $\tau$ is a column permutations of the columns of $\lambda$ for which $\tau\cdot \fS_P=\sigma\cdot \fS_P$ for some row permutation $\sigma$ of the boxes, and where $\rho$ is a row permutation for which $\rho\cdot \fS=\fS$.
\end{lemma}

\begin{proof}
First note that we are looking to compute the size of the equivalence class, among $T$-snakings, of a given $T$-snaking $P$ under the equivalence of generating the same monomial as $x_P^\fS$.
Given a row permutation $\rho$ for which $\rho\cdot \fS=\fS$, note that $$x_P^\fS=x_{\rho\cdot P}^{\rho \cdot \fS}=x_{\rho\cdot P}^\fS$$
and so $\rho\cdot P$ is in this equivalence class. 
It follows that we can choose a representative of the equivalence class of $P$ that has all of its entries in a given row that correspond to the same entry in $\fS$ sorted in increasing order from left to right. 
We call such a representative \textit{canonical} with respect to $\fS$, and we now can assume without loss of generality that our original $P$ was canonical.

Consider another $T$-snaking $Q$.  First, let $\rho$ be the unique row permutation that fixes $\fS$ under the box action and for which $P':=\rho\cdot Q$ is canonical.    Now, to obtain $\tau$ from $P'$, first note that there are unique row permutations $\pi,\pi'$ that send $P,\fS$ and $P',\fS$ respectively to obtain the pairs $T_P,\fS_P$ and $T_{P'},\fS_{P'}$ (because $P$ and $P'$ are standard).  We have $$x_{T_P}^{\fS_P}=x_P^\fS=x_{P'}^\fS=x_{T_{P'}}^{\fS_{P'}}$$
and $T_P$ and $T_{P'}$, by definition, have the same sets of column elements as $T$.  Thus there is a unique column permutation $\tau$ of $T$ such that $\tau\cdot T_P=T_{P'}$, and we compute that $$x_{T_P}^{\fS_P}=x_{\tau \cdot T_P}^{\tau \cdot \fS_P}=x_{T_{P'}}^{\tau\cdot \fS_P}.$$

Combining the last two equations, we find $x_{T_{P'}}^{\tau\cdot \fS_P}=x_{T_{P'}}^{\fS_{P'}}$.  Thus $\tau \cdot \fS_P$ has the same entries in each box as $\fS_{P'}$, in other words, $\tau\cdot \fS_P=\fS_{P'}$.  But $\fS_P=\pi \cdot \fS$ and $\fS_{P'}=\pi'\cdot \fS$, so $\tau \cdot \fS_P = \pi'\cdot \pi^{-1}\cdot \fS_P$, which means $\tau$ applied to $\fS_P$ is indeed a row permutation of $\fS_P$.  We have thus constructed a pair $(\tau,\rho)$ from a starting $T$-snaking $Q$.

We now show that we can reverse the process; fix column permutation $\tau$ that sends $\fS_P$ to a row permutation of itself, and choose a row permutation $\rho$ such that $\rho\cdot \fS=\fS$.  Define $\pi$ again to be the unique row permutation sending $P,\fS$ to $T_P,\fS_P$, and define $T_P'=\tau \cdot T_P$.

The row permutation $\sigma$ for which $\tau\cdot \fS_P=\sigma\cdot \fS_P$ is unique only up to a choice of row permutation that fixes $\fS_P$ under the box action; we let $\sigma$ be the unique such permutation such that if we define $P'=\pi^{-1}\sigma^{-1}\cdot T_P'$, then $P'$ is in canonical form with respect to $\fS$.  We claim that for this definition of $P'$, we have  $x_{P'}^\fS=x_{P}^\fS$.  Indeed, we have 
$$x_{P'}^\fS=x_{\pi^{-1}\sigma^{-1}\cdot T_P'}^{\pi^{-1}\cdot \fS_P}=x_{\sigma^{-1}\tau \cdot T_P}^{\sigma^{-1}\tau \cdot  \fS_P}=x_{T_P}^{\fS_P}=x_{P}^{\fS}.$$
Thus we have used $\tau$ to uniquely determine a choice of $P'$ in canonical form, and using $\rho$ to form $Q=\rho^{-1}\cdot P'$, we obtain a generic $Q$ for which $x_Q^\fS=x_P^\fS$.  This reverses the process above, and so we have a bijection.
\end{proof}

This leads us to the following definitions.

\begin{definition}
    The \textbf{$\rho$-group} $H(\fS)$ of a filling $\fS$ of shape $\lambda$ is the group of row permutations $\rho$ for which $\rho\cdot \fS=\fS$.  Note that $H(S)$ is a subgroup of the row group $\mathcal{R}(T)$, where $T$ is the canonical box ordering $T$ of shape $\lambda$, and $H(\fS)\cong H(\pi\cdot \fS)$ for any row permutation $\pi$.
\end{definition}

\begin{definition}
    The \textbf{(alternating) $\tau$-set} $A(\fS)$ of a filling $\fS$ of $\lambda$ is the set of column permutations $\tau$ of $T$, where $T$ is the canonical box ordering of shape $\lambda$, whose box action on $\fS$ preserves the content of each row of $\fS$.
\end{definition}

\begin{remark}
    The alternating $\tau$-set is not a group in general. 
 Suppose we have  $$\fS=\young(100,0112).$$  Using \raisebox{-0.1cm}{\scriptsize $\young(246,1357)$} as the canonical box labeling as usual, we have $A(\fS)=\{\mathrm{id},(12)(34),(12)(56)\}$ which is not a subgroup - indeed, the composition $(12)(34)(12)(56)=(34)(56)$ is not an element of $A(\fS)$.
\end{remark}

\begin{theorem}\label{thm: coeff}
    The coefficient of $x_P^\fS$ in $F_T^\fS$ is $$\pm |H(\fS)|\cdot \sum_{\tau \in A(\fS_P)}(-1)^\tau$$
\end{theorem}

\begin{proof}
    This follows immediately from Lemma \ref{lem: same-monomial-rowcol}.
\end{proof}

    All of the possible tableaux $\fS_P$ are row permutations of $\fS$, so we also obtain the following condition for when the polynomial is $0$.

\begin{corollary}
    The higher Specht polynomial $F_T^\fS$ is identically $0$ if and only if for all row permutations $\fS'$ of $\fS$, we have $\sum_{\tau \in A(\fS')}(-1)^\tau=0$.
\end{corollary}

Note that the sets $A(\fS_P)$ may have different sizes and so the coefficients of different monomials in the higher Specht polynomial may be different as well, even when all of the elements $\tau$ are even.

\begin{example}
    In the example above, we have $|H(\fS)|=4$, and there are only even permutations in $A(\fS)$ with $|A(\fS)|=3$, so the coefficient of $x_T^\fS=x_2x_3x_5x_7^2$ in $F_T^\fS$ is $4\cdot 3=12$.  However, if we consider the $T$-snaking $$P=\young(246,1375),$$ then $$\fS_P=\young(100,0121)$$ and we now have $A(\fS_P)=\{\mathrm{id},(12)(34)\}$ so the coefficient of $x_P^\fS=x_2x_3x_7x_5^2$ is $4\cdot 2=8$.
\end{example}

We now examine cases in which there are odd permutations $\tau \in A(\fS_P)$.

\begin{example}
 Suppose $$\fS=\young(1200,0122)$$ and we again use the standard box labeling $T=\young(2468,1357)$. Then $$A(\fS)=\{\mathrm{id},(12)(34)(56),(12)(34)(78)\}$$   which has two odd permutations and one even.  So, there is an odd permutation but the coefficient does not die.  In fact, the coefficient of $x_T^\fS$ in $F_T^\fS$ is $$4(1-1-1)=-4.$$
\end{example}

\begin{corollary}\label{cor:odd}
    If there are no odd permutations in $A(\fS_P)$ for some $T$-snaking $P$ then $F_T^\fS$ is nonzero.
\end{corollary}

It was already known \cite{HigherSpechtDiagonal} that if $\fS$ is semistandard then $F_T^\fS$ is nonzero, but our result gives a new immediate proof of this fact.

\begin{corollary}\label{cor: ss nonzero}
    If $\fS$ is semistandard, then $F_T^\fS$ is nonzero.  Moreover, the coefficient of $x_T^\fS$ is positive and is equal to $|H(\fS)|$.
\end{corollary}

\begin{proof}
    For semistandard $\fS$, the alternating $\tau$ set $A(\fS)$ consists only of the identity element, and so it has no odd permutations.  By Corollary \ref{cor:odd} and Theorem \ref{thm: coeff} the result follows.
\end{proof}

Note also that cocharge being greedy (Lemma \ref{lem:greedy}) means that in many cases, cocharge tableaux are semistandard. However, this is not always the case (Section \ref{sec:counterexample}).

\subsection{Two column case}

We now give a necessary and sufficient condition for $F_T^\fS$ to be nonzero when the shape of $\fS$ and $T$ has only two columns.  We start with the following lemma.

\begin{lemma}\label{lem:column-cancel}
    If $\fS_P$ has two identical entries in the same column, then the coefficient of $x_P^\fS$ in $F_T^\fS$ is zero.
\end{lemma}

\begin{proof}
 Let $i$ and $j$ be the boxes (as in the standard column labeling) of the two identical entries.  Consider the transposition $(ij)$.  For any $\tau\in A(\fS_P)$, we also have $\tau\circ (ij) \in A(\fS_P)$ since $(ij)\cdot \fS_P=\fS_P$.  Thus we can match each $\tau$ with $\tau\circ (ij)$ to pair off the entries of $A(\fS_P)$ into pairs of elements with opposite signs; this shows that the coefficient is $0$. 
\end{proof}

\begin{definition}
    We say that $\fS$ is \textbf{column-separable} if there is a row permutation $\sigma$ such that $\sigma\cdot \fS$ has no pair of equal entries in the same column.
\end{definition}

\begin{proposition}
    Let $\fS$ be a tableau with two column shape $\lambda$. Then for any $T\in \SYT(\lambda)$, $F_T^\fS \neq 0$ if and only if $\fS$ is column-separable.
\end{proposition}

\begin{proof}
    If $\fS$ is not column-separable, then Lemma \ref{lem:column-cancel} shows that $F_T^\fS$ is identically $0$.  

    If $\fS$ is column-separable, we may without loss of generality assume that $\fS$ itself has no repeated entries in the same column.   We show that any $\tau\in A(\fS)$ is even, and by Corollary \ref{cor:odd} this will imply that $F_T^\fS\neq 0$.  Indeed, any such $\tau$, acting on $\fS$ under the box action, must fix every element of the first column in a row higher than the height of the second column, in order to preserve the row sets, since the first column of $\fS$ has all distinct entries.  Thus $\tau$ can only nontrivially permute entries in the rows that have two boxes.

    Now consider a cycle of $\tau$ in the first column, say, $(b_1\, b_2\,\cdots\, b_k)$.  If $b_1',b_2',\ldots,b_k'$ are the corresponding entries to the right of $b_1,\ldots,b_k$ respectively in the second column, then in order for $\tau$ to preserve the row sets of $S$ we must have the cycle $(b_1'\, b_2'\,\cdots\, b_k')$ in $\tau$ as well, since the elements of the second column of $\fS$ are distinct as well.  Thus we have a copy of every cycle in the first row, in the second row, and vice versa by a similar argument.  It follows that $\tau$ is even, as desired.
\end{proof}

\begin{example}
    Let
    $$
    \fS_P = 
    \begin{ytableau}
        *(pink) 5 & *(pink)3 \\
        *(pink)4 & *(pink)0 \\
        *(pink)3 & *(pink)4 \\
        2 & 1 \\
        1 & 6 \\
        *(pink)0 & *(pink)5
    \end{ytableau}
    \hspace{0.5cm}
    \fS_{P'} = 
    \begin{ytableau}
        *(pink)3 & *(pink)5 \\
        *(pink)0 & *(pink)4 \\
        *(pink)4 & *(pink)3 \\
        2 & 1 \\
        1 & 6 \\
        *(pink)5 & *(pink)0
    \end{ytableau}
    $$
    where we have highlighted the row transpositions, which are contained in rows $1,4,5,6$. We have that $\tau^{(1)} = (1 3 2 4)$ and $\tau^{(2)} = (1 4 2 3)$, so that $\tau^{(1)} = (\tau^{(2)})^{-1}$.
\end{example}

\subsection{Column separability is not sufficient in general}

Though column separability gives a clean necessary and sufficient condition for higher Specht polynomials to be nonzero for two column shapes, the result does not extend to more than two columns. Suppose $\fS$ is $$\young(012,012)$$ and $T$ is any standard tableau of the same shape. Then $F_T^\fS=0$, but $\fS$ is column separable, since it has the row permutation below: $$\young(012,120)$$

\section{Proof of Conjecture \ref{conj: new higher specht} for three row and hook shapes}\label{sec:proof}

 In this section, we resolve Conjecture \ref{conj: new higher specht} in the case where $\mu$ is a hook shape or a three row partition.  That is, we show that $\BBB_\mu$ is a basis for $R_\mu$ in these cases.

\subsection{Hook shapes}

In \cite{HigherSpechtDiagonal}, a basis for hook-shape Garsia-Haiman modules was given as follows.

\begin{definition}
    Let $\mu = (n-k+1,1^{k-1})$. We define the $\mu$-cocharge tableau of $S\in \SYT_n$, denoted $\ccTab_\mu(S)$,  to be the tableau of the same shape where 
    \begin{itemize}
        \item we have 0 in the squares occupied by $1,2,\dots, n-k+1$ in $S$
        \item the rest of the squares are given by the cocharge labeling on $n-k+1,\dots, n$.
    \end{itemize}
\end{definition}

\begin{definition}
    Let $\mu = (n-k+1,1^{k-1})$. We define the reverse $\mu$-cocharge tableau of $S\in \SYT_n$,  denoted $\ccTab'_\mu(S)$, to be the tableau of the same shape where 
    \begin{itemize}
        \item we have 0 in the squares occupied by $n-k+1,\dots, n$ in $S$
        \item for $n-k+1,\dots, 1$, we calculate the cocharge in reverse.
    \end{itemize}
\end{definition}

Then we define $\mathbf{xy}_T^S = \mathbf{x}_T^{\ccTab_\mu(S)}\mathbf{y}_T^{\ccTab'_\mu(S)}$ and $F_{T}^S(\mathbf{x},\mathbf{y}) = \varepsilon_T \mathbf{xy}_T^S$. Note that all of these definitions depend on $\mu$.

\begin{theorem}[\cite{HigherSpechtDiagonal}]
    Let $\mu = (n-k+1,1^{k-1})$. Then the following set is a higher Specht basis for $\mathrm{DR}_\mu:$
    \begin{align}\label{eq: higher specht hook}
            \{F_{T}^S(\mathbf{x},\mathbf{y}): S,T\in \SYT_n, \sh(S) = \sh(T)\}
    \end{align}

\end{theorem}

We can make the following observation:
\begin{lemma}\label{lem: GH at y=0}
    Let $\mu = (n-k+1,1^{k-1})$. We have that $F_{T}^S(\mathbf{x},0) \neq 0$ if and only if $\ctype(S)\unrhd \mu$.
\end{lemma}
\begin{proof}
    Note that $F_{T}^S(\mathbf{x},0) \neq 0$  if and only if $\ccTab'_\mu(S)  = 0$, which means $1,\dots, n-k+1$ all appear in the same row of $S$. This is equivalent to saying the first row of $S$ contains $1,\dots, n-k+1$, thus $S$ is $\mu$-catabolizable and $\ctype(S)\unrhd \mu$.
\end{proof}

A direct corollary of this is the following:
\begin{corollary}\label{cor:hooks}
    Conjecture \ref{conj: new higher specht} is true for hook shapes.
\end{corollary}

\begin{proof}
    From Lemma \ref{lem: GH at y=0}, we know that the following set is a basis of $R_\mu$ for $\mu =  (n-k+1,1^{k-1})$.
    \[\{F_{T}^S(\mathbf{x},0): S,T\in \SYT_n, \sh(S) = \sh(T), \ctype(S)\unrhd \mu\}.\]

    In this case, we have $F_{T}^S(\mathbf{x},0) = \varepsilon_T\mathbf{x}_T^{\ccTab_\mu(S)}$. Note that if $\ctype(S)\unrhd \mu$, then $1,\dots, n-k+1$ all appear in the first row of $S$, thus $\cc(S) = \ccTab_\mu(S)$. Thus  $F_{T}^S(\mathbf{x},0) = \varepsilon_T\mathbf{x}_T^{\ccTab_\mu(S)} =   \varepsilon_T\mathbf{x}_T^{\cc(S)} = F_T^S$, and we precisely cover the elements of $\BBB_\mu$.  
\end{proof}

\subsection{Three rows}

To prove Conjecture \ref{conj: new higher specht} for three row shapes, we first prove a lemma classifying the tableaux with catabolizability type $\mu=(n,m,r)$.

\begin{lemma}\label{lem:three-row-cat}
   Let $n\ge m\ge r\ge 0$.  There are exactly $r+1$ tableaux $S$ with catabolizability type $\mu=(n,m,r)$.  They are formed by placing $1,\ldots,n$ in the bottom row, $n+1,\ldots,n+m$ in the second row, $n+m+1,\ldots,n+m+i$ in the bottom row for some $i$, and the remaining numbers up in the third row.
\end{lemma}

\begin{example}
    For $\mu=(5,3,3)$, there are four tableaux $S$ with catabolizability type $\mu$, namely 
    $$\young(9{{10}}{{11}},678,12345) \hspace{1cm} \young({{10}}{{11}},678,123459)\hspace{1cm} \young({{11}},678,123459{{10}})\hspace{1cm} \young(678,123459{{10}}{{11}}).$$
    with cocharge tableaux
    $$\young(222,111,00000) \hspace{1cm} \young(22,111,000001)\hspace{1cm} \young(2,111,0000011)\hspace{1cm} \young(111,00000111).$$
\end{example}

\begin{proof}
  If $\ctype(S)=\mu=(n,m,r)$, we must have $1,\ldots,n$ on the bottom row, and $n+1$ on the second row, so that there are exactly $n$ cocharge labels of $0$ for Algorithm \ref{alg: ctype} to produce a partition with first row of size $n$.  Then, we claim $n+1,\ldots,n+m$ must all be on the second row (with cocharge $1$) - if one is in the third row, it will have cocharge $2$ and we will not have enough boxes on the second row of the ctype algorithm output.  If one is in the bottom row, it will also have cocharge $1$ but will increment to $2$ by the time we use it in the algorithm, again leaving us with the wrong number of boxes in the second row.   We also note that $n+m+1$ is not in the second row, since otherwise the second row of $\ctype(S)$ will be at least $m+1$.

  Since we cannot have a cocharge value of $3$, the remaining values $n+m+1,\ldots,n+m+r$ have cocharge $1$ or $2$.  Thus there is some $i$ for which $n+m+1,\ldots,n+m+i$ are on the bottom row (with cocharge $1$) and then $n+m+i+1,\ldots,n+m+r$ are placed left to right in order (but possibly skipping some columns) starting on the second or third row, each with cocharge $2$.  Thus in particular $S$ has at most $3$ rows.
  
   If $n+m+i+1$ is in the second row, then it has cocharge $2$, and by Algorithm \ref{alg: ctype} rule we would encounter the $2$ before having any boxes in the second row, which increments it to $3$, a contradiction.  Thus it is in the third (top) row.   A similar argument now shows that  the remaining entries $n+m+i+2,\ldots,n+m+r$ are also all in the top row, and the result follows.
\end{proof}

\begin{lemma}\label{lem:three-row-distinct}
    Let $\nu$ and $\mu$ be distinct shapes with at most three rows.  Then if $\ctype(S)=\nu$ and $\ctype(S')=\mu$, then $S$ and $S'$ cannot have the same shape and also the same total cocharge.
\end{lemma}

\begin{proof}
    This follows from Lemma \ref{lem:three-row-cat}; any $S$ of a fixed shape whose catabolizability type has at most three rows has a cocharge tableau  with only $0$s and $1$s in the bottom row, only $1$s in the second, and only $2$s in the top, and the total cocharge determines how many $1$s are in the bottom row.  We can then reconstruct $S$ uniquely from the cocharge tableau.  We can also reconstruct its catabolizability type uniquely - the first part is the number of $0$s, the second part is the length of the second row, and the third part is the total number of remaining squares.  Thus we can reconstruct the catabolizability type uniquely from the shape and cocharge.
\end{proof}

We now outline the proof of Conjecture \ref{conj: new higher specht} for three rows, noting that a key step---proving that the conjectured basis elements are not identically zero in the quotient $R_\mu$---will be proven in Section \ref{sec:nonzero}.

\begin{proposition}
   If $\mu=(n,m,r)$ has three rows, then $\BBB_\mu$ is a basis of $R_\mu$.
\end{proposition}

\begin{proof}
    We proceed by double induction, first on $r$, and then on $m$.  For the base case, $r=0$, we are in the two row case, which is known \cite{ChouHanada}.  Now fix $r\ge 1$ and suppose the proposition is true for all partitions with third part $\mu_3<r$, and consider the partitions of the form $(n,m,r)$.  We show that the conjecture holds for these partitions by induction on $m$, starting with the base case $m=r$.

    For $\mu=(n,r,r)$, the only partition that covers $\mu$ in dominance order is $\nu=(n,r+1,r-1)$ (or $\nu=(r+1,r,r-1)$ if $n=r$ as shown at right below). By the induction hypothesis, the proposition holds for $\nu$.
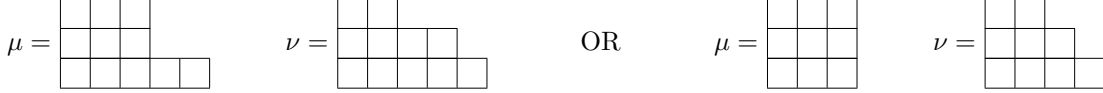
    $$\raisebox{0.5cm}{$\mu =$ } \young(\,\,\,,\,\,\,,\,\,\,\,\,) \hspace{1cm} \raisebox{0.5cm}{$\nu=$ }\young(\,\,,\,\,\,\,,\,\,\,\,\,) \hspace{1cm}\raisebox{0.5cm}{ \text{ OR } }\hspace{1cm}\raisebox{0.5cm}{$\mu =$ } \young(\,\,\,,\,\,\,,\,\,\,) \hspace{1cm} \raisebox{0.5cm}{$\nu=$ }\young(\,\,,\,\,\,,\,\,\,\,)$$
    Then $\BBB_\nu$, which is a basis of $R_\nu$ by the induction step, pulls back to an independent set in $R_\mu$, spanning the same set of $\mathfrak{S}_n$-modules in the same degrees as in $R_\nu$.  We now consider the extra elements of the conjectured basis for $\mu$.  These arise from tableaux $S$ for which $\ctype(S)\dom \mu$ but $\ctype(S)\not\dom\nu$.  Since $\nu$ is the only cover of $\mu$ in dominance order (as covers are formed by moving a corner box from a higher row to the next available lower row), this means $\ctype(S)=\mu$.

    Thus we are looking at the tableaux $S$ whose catabolizability type is $(n,r,r)$.  By Lemma \ref{lem:three-row-distinct}, no tableau $S'$ with the same shape and degree as one of the new $S$ possibilities above can be in $\mathcal{B}_\nu$. So, the resulting higher Specht modules from all $F_T^S$ for each new $S$ are all distinct from each other and from any element of the $R_\nu$ basis.  Since they span distinct $\mathfrak{S}_n$-modules in distinct degrees, it follows that they are independent from each other and the $R_\nu$ basis, and do not die in $R_\mu$ by the work in Section \ref{sec:nonzero} below.  This completes the proof of the base case.

    Now let $m$ be arbitrary and suppose the claim holds for all partitions with a smaller second part than $m$ (with fixed third part $r$).   Let $\mu=(n,m,r)$.  We again have two possibilities based on whether $n=m$ or not, and if so we have one $\nu$ that covers $\mu$ in dominance order, and if not we have two:
    $$\raisebox{0.5cm}{$\mu =$ } \young(\,\,\,,\,\,\,\,\,,\,\,\,\,\,) \hspace{1cm} \raisebox{0.5cm}{$\nu=$ }\young(\,\,\,,\,\,\,\,,\,\,\,\,\,\,)
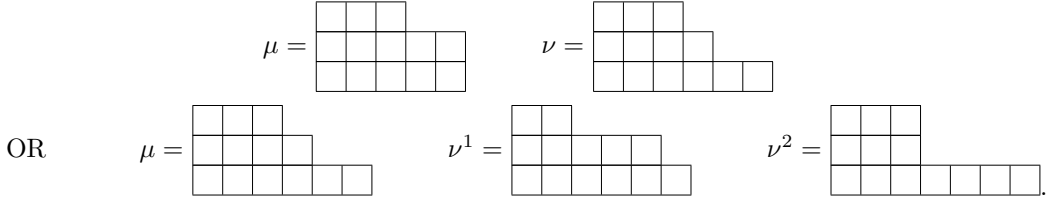
$$
    $$\raisebox{0.5cm}{ \text{ OR } }\hspace{1cm}\raisebox{0.5cm}{$\mu =$ } \young(\,\,\,,\,\,\,\,,\,\,\,\,\,\,) \hspace{1cm} \raisebox{0.5cm}{$\nu^1=$ }\young(\,\,,\,\,\,\,\,,\,\,\,\,\,\,)\hspace{1cm} \raisebox{0.5cm}{$\nu^2=$ }\young(\,\,\,,\,\,\,,\,\,\,\,\,\,\,).$$ Note that in the second case above, the partitions $\nu^1$ and $\nu^2$ are formed by moving one corner box of $\mu$ to the next row below.

    In the first case, the proof proceeds as above, comparing $R_\mu$ to $R_\nu$, and since both are three row shapes we are again done by Lemma \ref{lem:three-row-distinct} and Section \ref{sec:nonzero} below.

    We now analyze the second case, and consider the two quotient maps $R_\mu\to R_{\nu^1}$ and $R_\mu \to R_{\nu^2}$. Then the sets $\BBB_{\nu^1}$ and $\BBB_{\nu^2}$ are bases of $R_{\nu^1}$ and $R_{\nu^2}$ by the induction hypotheses on $r$ and $m$ respectively, and pull back to independent sets in $R_\mu$, with overlap coming from tableaux $S$ whose ctype dominates both $\nu^1$ and $\nu^2$.  We claim that the union of these bases in $R_\mu$ is also independent.  Indeed, Lemma \ref{lem:three-row-distinct} shows that we cannot have an $S$ of ctype $\nu^1$ that matches the shape and degree from anything in $\mathcal{B}_{\nu^2}$, and conversely for those of ctype $\nu^2$ being distinct from those in $\mathcal{B}_{\nu^1}$.  Finally, the new tableaux with catabolizability type $\mu$ in $\mathcal{B}_\mu$ are also distinct from those in the $\nu^1$ and $\nu^2$ bases again by Lemma \ref{lem:three-row-distinct}, and we are done as before. 
\end{proof}

\section{Nonvanishing of the basis in the quotient}\label{sec:nonzero}
We now show that our higher Specht basis of the Garsia-Procesi ring $R_\mu$ does not vanish in $R_\mu$ for $\mu$ having three rows.
 To do this, we will show the following: 
\begin{proposition}\label{thm: F_T^S nonzero}
    If $\ctype(S) = \mu$, there exists a filling $U\in \Tab(\mu)$ such that $F_T^S(\partial)g_{U} \neq 0$. 
\end{proposition}
Using Proposition \ref{thm: F_T^S nonzero} along with Corollary \ref{cor: nonzero col strict}, we get the following result.

\begin{theorem}
    If $S,T\in \SYT_n$ with the same shape and $\ctype(S)\unrhd \mu$, we have that $F_T^S\not\in I_\mu$. 
\end{theorem}
\begin{proof}
Consider $S\in \SYT_n$ with $\ctype(S)  = \lambda\unrhd \mu$. By Proposition \ref{thm: F_T^S nonzero} and Corollary \ref{cor: nonzero col strict}, we have that $F_T^S\not\in I_\lambda$. 
Since $\lambda \unrhd \mu  \Leftrightarrow I_\mu \subset I_\lambda$, this implies that $F_T^S\notin I_\mu$.
\end{proof}

The rest of section will be devoted to proving Proposition \ref{thm: F_T^S nonzero}.
Throughout we assume $S,T \in \SYT_n$ with $\ctype(S) = \mu$.
First, we check the coefficient of $\xx_T^S$ in $F_T^S$ using results from previous sections.
Since $S\in \SYT_n$, we know that $\cc(S)$ is semistandard.
 An immediate corollary of Corollary \ref{cor: ss nonzero} is the following.

\begin{corollary}
    The coefficient of $\mathbf{x}_T^S$ in $F_T^S$ is $|H(\cc(S))|$.
\end{corollary}

\subsection{Constructing \texorpdfstring{$U$}{}}
Now, we will construct a standard filling $U$ of shape $\mu$ such that $\partial_T^S g_U \neq 0$, where $\partial_T^S$ denotes the differential operator obtained by replacing $x_i \mapsto \partial_i$ in $\mathbf{x}_T^S$.  More precisely:

\begin{definition}
    We write $\partial_T^S$ as a shorthand for $x_T^S(\partial)$.
\end{definition}

To do this, we modify Algorithm \ref{alg: ctype} to form a filling $U$ rather than just a partition shape $\nu$ in the following way.

\begin{definition}
    Let $S,T$ be standard Young tableaux of the same shape.  Let $x=\cc(\rw(S))$, and let $U=\emptyset$ be the empty Young tableau.  The \textbf{U-tableau} $U(S,T)$ of the pair $(S,T)$ is formed by modifying Algorithm \ref{alg: ctype} in the following way. We fill the square added when reading $c_i$ by the label $t$ from $T$ that appears in the same box as $c_i$ occurs in $S$.

We define $U(S,T)=U$ to be the output of this algorithm, and note that it is a standard filling of shape $\ctype(S)$.
\end{definition}

\begin{example}\label{ex: cat alg}
Consider \[S = \ytableaushort{6,34,125}, \hspace{2cm} \cc(S) = \ytableaushort{2,11,001}, \hspace{2cm}  T = \ytableaushort{3,25,146}\]
Then we have \begin{align*}
        \cc(\rw(S))=   2 \ 1 \ 1 \ 0 \ 0 \ 1 .
    \end{align*}
    We apply $f$ to $(\cc(w),\emptyset)$ repeatedly until we get an empty word in the first coordinate. 
The position we are reading at each step of the algorithm is underlined. 
The first step is
  \ytableausetup{smalltableaux, centertableaux}
$$( 2 \ 1 \ 1 \ 0 \ 0 \ \mathbf{\underline{1}} \hspace{.1cm}, \hspace{.3cm} \emptyset) \to 
     ( 2 \ 1 \ 1 \ 0 \ \mathbf{\underline{0}} \ 2\hspace{.1cm} , \hspace{.3cm}\emptyset \hspace{.1cm}).$$
Now at the next step, we add a box to $U$. The $0$ we read corresponds to the box that contains $4$ in $T$, hence we add $4$ to $U$.

$$ ( 2 \ 1 \ 1 \ 0 \ \mathbf{\underline{0}} \ 2\hspace{.1cm}, \hspace{.3cm}\emptyset \hspace{.1cm}) \to ( 2 \ 1 \ 1 \  \mathbf{\underline{0}} \ \hspace{.2cm} \ 2 \hspace{.1cm}, \hspace{.3cm} \ytableaushort{4} \hspace{.1cm}).$$
Continuing, we have:
\begin{align*}
    &( 2 \ 1 \ 1 \  \mathbf{\underline{0}} \ \hspace{.2cm} \ 2 \hspace{.1cm}, \hspace{.3cm} \ytableaushort{4} \hspace{.1cm}) \to  ( 2 \ 1 \ \mathbf{\underline{1}} \  \hspace{.2cm} \ \hspace{.2cm} \ 2 \hspace{.1cm},\hspace{.1cm} \ytableaushort{41}\hspace{.1cm}) \to ( 2 \ \mathbf{\underline{1}} \  \hspace{.2cm}  \  \hspace{.2cm} \ \hspace{.2cm} \ 2\hspace{.1cm}, \hspace{.3cm}\ytableaushort{5,41}\hspace{.1cm}) \\ &\to   ( \mathbf{\underline{2}} \ \hspace{.2cm} \  \hspace{.2cm}  \  \hspace{.2cm} \ \hspace{.2cm} \ 2 \hspace{.1cm}, \hspace{.3cm}\ytableaushort{52,41}\hspace{.1cm}) \to ( \hspace{.2cm}\ \hspace{.2cm} \  \hspace{.2cm}  \  \hspace{.2cm} \ \hspace{.2cm} \ \mathbf{\underline{2}}  \hspace{.1cm},\hspace{.3cm} \ytableaushort{3,52,41}\hspace{.1cm}) \to
     ( \hspace{.2cm}\ \hspace{.2cm} \  \hspace{.2cm}  \  \hspace{.2cm} \ \hspace{.2cm} \ \hspace{.2cm} \hspace{.1cm}, \hspace{.3cm}\ytableaushort{36,52,41}\hspace{.1cm}) 
\end{align*}
Thus we conclude $U = \ytableaushort{36,52,41}$ and $\ctype(634125) = (2,2,2)$.
\end{example}

In what follows, we use the following notation: if $b$ is a label in $U$, it also corresponds to an entry $c\in \cc(S)$ that was encountered at the step of the algorithm that added $b$ to $U$.  In this case we write $c=\cc(S)(b)$.

\begin{remark}
    The construction of $U$ is well defined for any partition $\mu$; however, throughout the rest of Section \ref{sec:nonzero} we assume $\mu$ has three rows. 
\end{remark}

We will use the following facts about $U$.

\begin{lemma}\label{lem: facts}
For a three row shape $\mu$ and a pair of SYTs $S,T$ with $\ctype(S)=\mu$, the following hold for the tableau $U=U(S,T)$.
\begin{enumerate}
    \item \label{lemma: box adding} If $b$ appears in row $i$ of $U$, then $\cc(S)(b) \leq i$.
    \item  \label{lem: cocharge consec} If $b_1, b_2$ appear consecutively in a column in $U$ (where $b_1$ is lower), we must have $\cc(S)(b_2) = \cc(S)(b_1)$ or $\cc(S)(b_1)+1$.

    \item \label{lem: cocharge 0}
    Let $\mathbf{b}$ be a column in $U$. Then the monomial $\xx_T^S$ has exactly one variable from column $\mathbf{b}$ with exponent 0. In particular, if $b_1,b_2$ are the entries in rows 0 and 1 of $\mathbf{b}$, then $(\xx_T^S)_{b_1,b_2} = x_{b_1}^0 x_{b_2}^1$.

    \item \label{lem: row2} If $c$ appears in row 2 of $T$, it appears in row 2 of $U$ as well.
\end{enumerate}
    
\end{lemma}

\begin{proof}
The first property follows immediately from the definition of the $U$-tableau. 

For (2) and (4), Lemma \ref{lem:three-row-cat} shows that the cocharge values we read in $cc(S)$ start with the $1$s from the bottom row (which we increment to $2$ and do not yet add to $U$) and then the remaining $0$s, $1$s, and $2$s in order, corresponding to adding boxes to the first, second, and third row in $U$.  We then read the incremented values and add boxes to the third row of $U$.  Thus we have cocharge values $0$ in the bottom row of $U$, $1$ in the middle row, and $1$ or $2$ on the top.

Then, (3) follows from the definition of the algorithm combined with property (2).
\end{proof}

Now, let $\mathbf{x}_U := \prod x_i^{\mathrm{row}(i)}$ where again the rows are $0$-indexed. 

\begin{lemma}\label{lem: partial_TS xU}
    We have $\partial_{T}^S \mathbf{x}_U \neq 0$.
\end{lemma}
\begin{proof}
If $b$ appears in row $i$ of $U$, the exponent of $x_b$ in $\mathbf{x}_U$ is $i$.
By Lemma \ref{lem: facts}, we know that $\cc(S)(b)\leq i$. Thus $\partial_b$ appears with exponent at most $i$ in $\partial_{T}^S.$ Since this holds for any choice of $b$, the claim holds.
\end{proof}

\begin{corollary}\label{cor: existence}
    We have $\partial_T^S g_U(\xx)\neq 0$.  Moreover, $\partial_T^S \xx_U$ appears in $\partial_T^S g_U(\xx)$ with a positive nonzero coefficient.
\end{corollary}
\begin{proof}
  We know 
  \begin{align}\label{eq: garnir U}
      g_U(\xx) = \sum_{\tau\in \mathcal{C}(U)} (-1)^\tau \tau(\xx_U).
  \end{align}
  
It is clear that for any $\tau\neq \text{id}$ in $\mathcal{C}(U)$, we have $\tau(\xx_U) \neq \xx_U$, which implies $\partial_T^S \xx_U \neq \partial_T^S \tau(\xx_U)$. Thus from Lemma \ref{lem: partial_TS xU}, the term $\partial_T^S \xx_U$ appears in $\partial_T^S g_U(\xx)$ with nonzero coefficient.
  \end{proof}

\subsection{Terms in \texorpdfstring{$F_T^S(\partial)g_U$}{}}
Note that using the $P$-snaking terminology, we can write
\[F_T^S  = \sum_{P\in \Snake(T)} (-1)^P \xx_P^S.\]

Combining this with \eqref{eq: garnir U}, we have
\begin{align}
    F_T^S(\partial) g_U(\xx) &=  \sum_{P\in \Snake(T)} (-1)^P \partial_P^S \left(\sum_{\tau\in \mathcal{C}(U)} (-1)^\tau\tau (\xx_U)\right) \\
    &=  \sum_{P\in \Snake(T)} \sum_{\tau\in \mathcal{C}(U)} (-1)^P(-1)^{\tau}\partial_P^S  \tau (\xx_U). \label{eq:formula}
\end{align}

To prove $F_T^S(\partial)g_U\neq 0$ (Proposition \ref{thm: F_T^S nonzero}), we will show the following.

\begin{theorem} \label{thm: +coeff}
    For $\mu$ a three-row partition, $\partial_T^S \xx_U$ appears in $F_T^S(\partial)g_U(\xx)$ with postive coefficient.
\end{theorem}

This theorem is the goal of the remainder of this section.  Our proof outline is as follows: From Equation \eqref{eq:formula}, it suffices to consider all combinations of $P,\tau$ such that $\partial_P^S  \tau (\xx_U) = \partial_T^S (\xx_U).$
 
    For such a pair, $\partial_P^S \tau\xx_U$ appears with sign $(-1)^P(-1)^\tau$ in $F_T^S(\partial)g_U(\xx)$. We will show below that in this case, we have that  $(-1)^P=(-1)^\tau$, hence any appearence of $\partial_P^S \tau\xx_U$ in $F_T^S(\partial)g_U(\xx)$ has positive sign.  Since there is at least one instance of the monomial by Corollary \ref{cor: existence}, the proof will be complete.

Our main tool towards proving $(-1)^P=(-1)^\tau$ is the following.

\begin{proposition}\label{prop: nonzero}
Given $\tau \in \mathcal{C}(U)$, if $\partial_P^S \tau (\xx_U) = \partial_T^S (\xx_U)$ for some $T$-snaking $P$, then $\tau( \partial_T^S (\xx_U)) =  \partial_T^S (\xx_U)$ and $\tau(\xx_{T}^S) = \xx_P^S$.  Moreover, $\tau$ fixes the elements of the top row of $U$ that have cocharge $1$ according to $S$.
\end{proposition}

\begin{remark}
    We conjecture that the first statement of Proposition \ref{prop: nonzero} is true for arbitrary shapes $\mu$, which would lead to a proof that all elements of $\BBB_\mu$ are nonzero in the quotient $R_\mu$.  However, here we only prove it for three row shapes.
\end{remark}

\begin{remark}
    In Proposition \ref{prop: nonzero}, it is helpful to keep in mind that $\partial_T^S x_U$, by the analysis in Lemma \ref{lem: facts}, is equal to a constant times the product $x_{a_1}\cdots x_{a_k}$ where $a_1,\ldots,a_k$ are the elements in the bottom row of $T$ at the right end with $\cc(S)(a_i)=1$.    Thus, for instance, the statement $\tau \partial_T^S(x_U)=\partial_T^S(x_U)$ is equivalent to $\tau$ fixing the set of entries $\{a_1,\ldots,a_k\}$, which all occur in the top row of $U$.  Since $\tau$ is a column permutation of $U$, this is equivalent to the statement that it fixes each entry $a_i$.
\end{remark}

To prove Proposition \ref{prop: nonzero}, we will use the following technical lemma.

\begin{lemma}\label{lem: perms}
   There is exactly one $3\times 3$ $\mathbb{N}$-matrix $A$ that satisfies the following:
   \begin{enumerate}[label=(\roman*)]
   \item $A_{2,3} = A_{3,2}=A_{3,3} = 0$
   \item  $A_{i,1} + A_{i,2} + A_{i,3}$ is the number of variables in $\partial_T^S \xx_U$ that appear with exponent $i-1$
   \item  $A_{1,j} + A_{2,j} + A_{3,j}$  is the number of values in $\cc(S)$ with cocharge $j-1$
   \item  $A_{1,1} = \mu_1$, $A_{2,1} + A_{1,2} = \mu_2$, $A_{3,1} + A_{2,2}+ A_{1,3} = \mu_3$.
   \end{enumerate}
\end{lemma}

\begin{proof}
Define $A$ to be a $3\times 3$ $\mathbb{N}$-matrix where 
\[A_{i,j} = \# \text{ of variables  which appear with exponent } i-1 \text{ in } \partial_T^S\xx_U  \text{ and exponent } j-1 \text{ in } \xx_T^S.\]
It is clear that this matrix satisfies conditions (i),(ii),(iii). To see that it satisfies (iv), note that
\[A_{1,1} = \# \text{ of variables  which appear with exponent } 0 \text{ in } \xx_U=\mu_1\]
\[A_{2,1} +A_{1,2}= \# \text{ of variables  which appear with exponent }  1 \text{ in }  \xx_U =\mu_2\]
\[A_{3,1} +A_{2,2}+A_{1,3}= \# \text{ of variables  which appear with exponent } 2 \text{ in }  \xx_U =\mu_3.\]

Thus $A$ satisfies the conditions.  For uniqueness, we can see that conditions (i), (ii) fix entry $A_{3,1}$. Since (iv) fixes entry $A_{1,1}$, we have the entry $A_{2,1}$ must also be fixed by (iii). Thus the first column is fixed.
Continuing in this manner, we see that the whole matrix is fixed.
\end{proof}

\begin{example}
    For the setting of Example \ref{ex: cat alg}, the matrix $A$ is 
$$\begin{pmatrix}
    2 & 2 & 1 \\
    0 & 1 & 0 \\
    0 & 0 & 0
\end{pmatrix}$$
Indeed, we have $x_U=x_2x_3^2x_5x_6^2$ and $x_T^S=x_2x_3^2x_5x_6$.  The monomial $\partial_T^S x_U$ is a constant times $x_6$, so the row sums are $5,1,0$ from top to bottom, the column sums are $2,3,1$ from left to right, and the diagonal sums are $2,2,2$ from top to middle.
\end{example}

\begin{remark}
   Note that Lemma \ref{lem: perms} does not naturally extend to more than 3 rows.
\end{remark}

\begin{proof}[Proof of Proposition \ref{prop: nonzero}]
Consider $\tau$ such that $\partial_T^S \xx_U = \partial_P^S \tau \xx_U$.   We can explicitly check how $\tau$ acts on each column, reducing the proof to the following claim.

 \textbf{Claim:} For each column $\mathbf{b}$ of $U$, we have that $(\tau\partial_T^S \xx_U)_{\mathbf{b}}  = (\partial_T^S\xx_U)_{\mathbf{b}} $ and $(\tau\xx_T^S)_{\mathbf{b}}  = (\xx_P^S)_{\mathbf{b}}$.

We prove this claim by cases based on the height of $\mathbf{b}$. 

\textbf{Case 1.} If $\mathbf{b}$ is of height one, we know that the first condition must hold. We also have that $(\partial_T^S \xx_U)_{\mathbf{b}} = (\partial_P^S \tau \xx_U)_{\mathbf{b}} = (\partial_P^S \xx_U)_{\mathbf{b}}$, which implies that $(\xx_{P}^S)_{\mathbf{b}}= (\xx_T^S)_{\mathbf{b}}= (\tau \xx_T^S)_{\mathbf{b}}$. Thus the claim holds.

\textbf{Case 2.} If $\mathbf{b} = \ytableaushort{{b_2}, {b_1}}$, then $(\xx_U)_{\mathbf{b}} = x_{b_1}^0 x_{b_2}^1$. Furthermore, from Lemma \ref{lem: cocharge 0}, we know that $(\partial_T^S)_{\mathbf{b}} = \partial_{b_1}^0 \partial_{b_2}^1$.   Thus $(\partial_{T}^S\xx_U)_{b_1,b_2} = x_{b_1}^0 x_{b_2}^0$, which implies the first claim holds for any $\tau$.

If $\tau(b_1) = b_2$, then $(\partial_{P}^S\tau \xx_U)_{b_1,b_2} = x_{b_1}^0 x_{b_2}^0$ implies that $(\xx_{P}^S)_{\mathbf{b}}= x_{b_1}^1 x_{b_2}^0 = (\tau \xx_T^S)_{b_1,b_2}$, thus the second claim holds. Otherwise, $\tau$ fixes the column and the second claim holds as well.

\textbf{Case 3.} Now, consider column $\mathbf{b} = \ytableaushort{{b_3}, {b_2}, {b_1}}$ of height 3. 
Then $(\xx_U)_{b_1,b_2,b_3} = x_{b_1}^0 x_{b_2}^1 x_{b_3}^2$. We know from  Lemmas \ref{lem: cocharge consec} and \ref{lem: cocharge 0}  that $(\partial_T^S)_{\mathbf{b}} = \partial_{b_1}^0 \partial_{b_2}^1\partial_{b_3}^2$ or $(\partial_T^S
)_{\mathbf{b}} = \partial_{b_1}^0 \partial_{b_2}^1\partial_{b_3}^1$.

If $(\partial_T^S)_{\mathbf{b}} = \partial_{b_1}^0 \partial_{b_2}^1\partial_{b_3}^2$, then as in Case 2 we have $(\partial_T^S \mathbf{x}_U)_{\mathbf{b}}=x_{b_1}^0x_{b_2}^0x_{b_3}^0$, and so applying any $\tau$ fixes this monomial, and the first claim holds.  For the second claim, note that $\partial_P^S \tau \xx_U=1$ as well, and so $\xx_P^S=\tau \xx_U=\tau \xx_T^S$ in this case. 

Now, if $(\partial_T^S)_{\mathbf{b}} = \partial_{b_1}^0 \partial_{b_2}^1\partial_{b_3}^1$, we have that $(\partial_P^S \tau \xx_U)_{\mathbf{b}}=(\partial_T^S\xx_U)_{\mathbf{b}} = x_{b_1}^0 x_{b_2}^0 x_{b_3}^1$.  Thus $\tau(b_1)$ cannot be $b_3$, since otherwise $\xx_{b_3}$ would have an exponent of $0$ in $\tau \xx_U$ (and hence in $\partial_P^S \tau \xx_U$).

Thus $\tau(b_1) = b_1$ or $b_2$.  Thus, we have the following possible subcases.
\begin{enumerate}[label=(\alph*)]
    \item $\tau$ fixes $\mathbf{b}$
    \item $\tau(b_1) = b_2$, $\tau(b_3) = b_3$.
    \item  $\tau(b_1) = b_1$, $\tau(b_2) = b_3$.
    \item  $\tau(b_1) = b_2$, $\tau(b_2) = b_3$.
\end{enumerate}

The claim is clearly true for (a). For (b), we can see that $\tau$ fixes $(\partial_T^S\xx_U)_{\mathbf{b}}$. We also have that $(\xx_P^S)_{\mathbf{b}} = x_{b_1}^1 x_{b_2}^0 x_{b_3}^1 = (\tau \xx_T^S)_{\mathbf{b}}$. 
Thus the claims hold.

Now for cases (c), (d), we can see that $(\xx_P^S)_{\mathbf{b}} = x_{b_1}^0 x_{b_2}^2 x_{b_3}^0$ (for (c)) and $(\xx_P^S)_{\mathbf{b}} = x_{b_1}^1 x_{b_2}^2 x_{b_3}^0$ (for (d)).  
We claim that cases (c), (d) are impossible. 
Assume that for our choice of $P,\tau$ there is at least one column $\mathbf{b}$ in $U$  that satisfies case (c) or (d).

Define a $3\times 3 \ \mathbb{N}$-matrix $\tilde{A}$ by
\[\tilde{A}_{i,j} =  \# \text{ of variables which appear with exponent } i-1 \text{ in } \partial_P^S\tau(\xx_U)  \text{ and exponent } j-1 \text{ in } \xx_P^S.\]

Since $\partial_T^S\xx_U = \partial_P^S \tau(\xx_U)$ and $\tau \xx_P^S = \xx_T^S$ for some permutation $\tau$,  we know that $\tilde{A}$ both satisfy the conditions outlined in Lemma \ref{lem: perms}.
However the existence of case (c) or (d) implies that
$\tilde{A}_{1,3} > A_{1,3}$ and $\tilde{A}_{1,2}< A_{1,2}$. By Lemma \ref{lem: perms}, we know that such $\tilde{A}_{1,3}$ cannot exist. Thus cases (c), (d) cannot happen.  It also follows that $\tau$ fixes the top row of $U$ (and hence of $T$ by Lemma \ref{lem: facts} (\ref{lem: row2})).
\end{proof}

In what follows, we call $\cc(S)(b)$ the ``cocharge'' of an element $b$ of $T$ or $U$.  We also implicitly use Lemma \ref{lem:three-row-cat} throughout, as the classification of possibilities for the tableaux $S$ and $T$ with catabolizability type $\mu$ having three rows.  We now prove some properties of $\tau$ and $P$ for the situation when $\partial_P^S \tau (\xx_U) = \partial_T^S (\xx_U)$.

\begin{lemma}\label{lem: fix height 1}
     Let $\tau \in \mathcal{C}(U)$ and $\partial_P^S \tau (\xx_U) = \partial_T^S (\xx_U)$ for some $T$-snaking $P$.  Then $\tau$ fixes the elements in columns of height $1$ in $T$.  Moreover, if $a_1,\ldots,a_k$ are the elements with cocharge $1$ at the end of the bottom row in $T$, they remain in the same set of boxes (possibly permuted amongst themselves) in $P$.
\end{lemma}

\begin{proof} 
    By Proposition \ref{prop: nonzero}, $\tau$ fixes the elements on the top row of $U$ which have cocharge $1$, which are precisely the elements at the right end of the bottom row of $T$ with cocharge $1$, and these are necessarily each columns of height $1$.  So $\tau$ fixes these rightmost columns of $T$.  Moreover, since $\tau x_T^S=x_P^S$ by Proposition \ref{prop: nonzero}, the elements are in the boxes  in the bottom row corresponding to cocharge $1$ in $P$ as well (since they cannot be in a higher row in $P$ by Lemma \ref{lemma: pigeonhole}).

If there are no columns of height $1$ besides these elements, we are done.  Otherwise we have $\mu_1\neq \mu_2$ and there are elements in columns of height $1$ of cocharge $0$. In this case, let $k$ denote the rightmost element of $T$ with cocharge $0$.   Then $k$ is in the bottom left corner of $U$, and then the element above $k$ in $U$ has cocharge $1$ and is not directly above $k$ in $T$.
   
    Assume for contradiction that $\tau$ does not fix $k$. Then $x_k$ has exponent $1$ or $2$ in $\tau x_T^S=x_P^S$ (where the equality is by Proposition \ref{prop: nonzero}).  However, since $P$ is a $T$-snaking, we know that $k$ must be in the bottom row of $P$ by Lemma \ref{lemma: pigeonhole}. Since $\tau$ fixes the entries in the bottom row of $T$ with cocharge 1, this implies that $x_k$ in fact has exponent $0$ in $x_P^S$, which is a contradiction. The same argument holds for all $b$ in bottom row in columns of height 1, and the result follows.
\end{proof}

\begin{lemma}\label{lem: fix bottom row}
    Let $\tau \in \mathcal{C}(U)$ and $\partial_P^S \tau (\xx_U) = \partial_T^S (\xx_U)$ for some $T$-snaking $P$.  If $\mu_1\neq \mu_2$, then $\tau$ fixes the entire bottom row of $T$, and its elements remain in the bottom row in $P$.
\end{lemma}

\begin{proof}
     Assume $\mu_1\neq \mu_2$, and let $r,s$ be the two elements above $k$ in $U$ (if they exist) in the second and third row respectively.
     For example, we can consider the following tableaux, where $(k,r,s) = (3,5,7)$:
     \begin{align}\label{ex: case 1}
         \cc(S) = \ytableaushort{2,11,0001}, \hspace{1cm} T = \ytableaushort{7,45,1236},\hspace{1cm} U = \ytableaushort{76,54,321}.
     \end{align}
  
     Note that $r$ is the rightmost element in the second row of $T$ and $s$ is the rightmost element in the top row.  
     Since $k$ is fixed, $r,s$ are either fixed or exchanged by $\tau$, so they must remain within the top two rows in the $T$-snaking $P$ as well, to have the same cocharge value in $x_P^S=\tau x_T^S$.      Let $b$ be the entry just below $r$ in $T$ (in the picture above, $b$ is $2$).  If $r,s$ are both in this column as well, then by the definition of $T$-snaking, $b$ remains in the bottom row in $P$.  Otherwise, if $s$ is to the left of $r$ and $b,r$ is a column of height $2$, by Lemma \ref{lemma: pigeonhole}, $r$ must be in row 2 of $P$, and again $b$ is in the bottom row of $P$ with cocharge $0$, and thus fixed by $\tau$.

     Inductively, we can remove $r,s$ from consideration and consider the next column of $U$, and by repeating the above argument we see that the entire bottom row of $T$ is fixed by $\tau$, and must remain in the bottom row in $P$. 
\end{proof}

\begin{lemma}\label{lem: fix top row}
     Let $\tau \in \mathcal{C}(U)$ and $\partial_P^S \tau (\xx_U) = \partial_T^S (\xx_U)$ for some $T$-snaking $P$.  If the top two rows of $T$ differ in length, then $\tau$ fixes the entire top row of $T$.
\end{lemma}

\begin{proof}     
    Consider again the rightmost elements $r,s$ of the second and third row respectively, so that the first column of $U$ is $k,r,s$ from bottom to top. 
    If $s$ is not fixed by $\tau$, either $r$ or $k$ replaces it, and so either $x_r$ or $x_k$ has exponent $2$ in $\tau x_T^S=x_P^S$.
    But since $r,k$ are in columns of height at most $2$, by Lemma \ref{lemma: pigeonhole}, it is impossible for $r,k$ to be in row 3 of $P$. Thus $\tau$ fixes $s$, and inductively we can move left and make the same argument for each element of the top row.
\end{proof}

We now show $\tau$ acts on the columns of $T$ as well as $U$.

\begin{proposition}\label{prop: tau in C(T)}
    Let $\tau \in \mathcal{C}(U)$ and $\partial_P^S \tau (\xx_U) = \partial_T^S (\xx_U)$ for some $T$-snaking $P$. Then $\tau \in \mathcal{C}(T)$. 
\end{proposition}

\begin{proof}
    We consider three cases.
    
    If $\mu_1\neq \mu_2$ and the top two rows of $T$ differ in length, then $\tau$ must be the identity by Lemmas \ref{lem: fix bottom row} and \ref{lem: fix top row} and is certainly in $\mathcal{C}(T)$ (as demonstrated in \eqref{ex: case 1}.

    Otherwise, suppose $\mu_1=\mu_2$ but the top two rows of $T$ differ in length. An example of this is given by the following three tableaux:
    \[S=\ytableaushort{22,111,000},\hspace{1cm} T=\ytableaushort{78,456,123}, \hspace{1cm} U=\ytableaushort{87,654,321}.\]
    
    Then the top row of $U$ is still fixed by $\tau$, and so $\tau$ is a product of transpositions in the bottom rows of $U$.  Moreover, since $\mu_1=\mu_2$, we see that two elements in the bottom two rows of $U$ are in the same column if and only if they are in the same column of $T$, and so $\tau\in \mathcal{C}(T)$.

    Now suppose $\mu_1\neq \mu_2$ but the top two rows of $T$ have the same length.  Note that in this case there are no elements of cocharge $1$ in the bottom row of $T$. Then the bottom row of $U$ is fixed by $\tau$, and so $\tau$ is a product of transpositions in the top two rows of $U$.  Since $T,U$ have the same pairs of elements in each column in the top two rows in this case, we again have $\tau \in \mathcal{C}(T)$.

  \[S=\ytableaushort{22,11,0000}, \hspace{1cm} T=\ytableaushort{78,56,1234}, \hspace{1cm} U=\ytableaushort{87,65,4321}\]

    The last case to consider is the special case in which $m=\mu_1=\mu_2$ and the top two rows of $T$ are equal in length; this implies $S,T,U$ are rectangular all of shape $\mu = (m,m,m)$ and the columns of $U$ are simply the columns of $T$ written right to left.  Thus $\tau\in \mathcal{C}(T)$ as desired.
\end{proof}

\begin{corollary}\label{cor: same sign}
   For $\tau,P$ as in Prop \ref{prop: nonzero}, we have $(-1)^\tau = (-1)^P$.
\end{corollary}

\begin{proof}
From Proposition \ref{prop: tau in C(T)}, we know that $\tau\in \mathcal{C}(T)$ such that $\tau(\xx_{T}^S) = \xx_P^S.$ 
We will show that $\tau_P = \tau$. Since $\tau\in \mathcal{C}(T)$, we know $\tau$ fixes any column of height 1. Furthermore, we know $\tau_P$ must also fix columns of height 1 by Lemma \ref{lemma: pigeonhole}. 

Now, consider columns of height 2 with entries $a,b$. Note that $a,b$ have exponents $0,1$ in $\xx_T^S$. We see that $\tau \xx_T^S = \xx_P^S$ implies that $\tau$ swaps $a,b$ if and only if $a,b$ have exponents $1,0$ in $\xx_P^S$. The only way this is possible is if $\tau_p$ swaps $a,b$ as well.

Similarly, consider a column with entries $a,b,c$ in $T$. Note that $a,b,c$ must have distinct exponents $0,1,2$ in $\xx_T^S$ by Lemma \ref{lem:three-row-cat}. Thus whatever $\tau$ does to this column, it is detectable in $\tau x_T^S=x_P^S$.  By Lemma \ref{lem: fix height 1}, the element of cocharge $1$ must be in the second row in $P$ and not in the bottom row, so $\tau_P$ must permute the column in the same way.  Thus $\tau = \tau_p$. 
\end{proof}

Theorem \ref{thm: +coeff} now follows, completing the proof that $\BBB_\mu$ is a basis for $R_\mu$ when $\mu$ has three rows.

\bibliography{refs}
\bibliographystyle{amsalpha}

\end{document}